\documentclass[11pt]{article}

\usepackage{amsmath, amsthm, amssymb}
\usepackage{enumitem}
\usepackage{mathtools}
\usepackage{etoolbox}
\usepackage[a4paper, margin=3cm]{geometry}
\usepackage{xcolor}
\usepackage[noadjust]{cite}
\usepackage{authblk}
\usepackage{nccmath}
\usepackage[normalem]{ulem}
\usepackage{hyperref}
\hypersetup{colorlinks=true, allcolors=teal}

\theoremstyle{plain}
\newtheorem{theorem}{Theorem}[section]
\newtheorem{definition}[theorem]{Definition}
\newtheorem{lemma}[theorem]{Lemma}
\newtheorem{proposition}[theorem]{Proposition}
\newtheorem{corollary}[theorem]{Corollary}

\theoremstyle{definition}
\newtheorem{example}[theorem]{Example}
\newtheorem*{remark}{Remark}

\setlist{noitemsep, topsep=1ex, parsep=1ex, partopsep=1ex}

\allowdisplaybreaks
\appto\normalsize{
	\abovedisplayskip=2ex plus 1ex minus 1ex
	\belowdisplayskip=2ex plus 1ex minus 1ex
	\abovedisplayshortskip=2ex plus 1ex minus 1ex
	\belowdisplayshortskip=2ex plus 1ex minus 1ex}
\appto\small{
	\abovedisplayskip=2ex plus 1ex minus 1ex
	\belowdisplayskip=2ex plus 1ex minus 1ex
	\abovedisplayshortskip=2ex plus 1ex minus 1ex
	\belowdisplayshortskip=2ex plus 1ex minus 1ex}

\AtBeginDocument{%
	\mathchardef\mathcomma\mathcode`\,
	\mathcode`\,="8000 
}
{\catcode`,=\active
	\gdef,{\mathcomma\discretionary{}{}{}}
}

\newcommand{\tr}{\operatorname{tr}}

\newcommand{\conv}{\operatorname{conv}}
\newcommand{\Ext}{\operatorname{Ext}}

\newcommand{\rank}{\operatorname{rank}}
\newcommand{\proj}{\operatorname{proj}}

\newcommand{\R}{\mathbb{R}}
\newcommand{\C}{\mathbb{C}}
\newcommand{\Tau}{{\cal T}}

\newcommand{\DS}{\operatorname{DS}}

\newcommand{\JF}{\operatorname{JF}}
\newcommand{\GJF}{\operatorname{GJF}}
\newcommand{\Rn}{\mathbb{R}^n}
\newcommand{\Rnp}{\mathbb{R}_+^n}
\newcommand{\Sn}{\mathcal{S}^n}
\newcommand{\Ln}{\mathcal{L}^n}

\newcommand{\Hn}{\mathcal{H}^n}
\newcommand{\V}{\mathcal{V}}
\newcommand{\W}{\mathcal{W}}

\newcommand{\one}{\mathbf{1}}
\newcommand{\lplus}{\Lambda_{+}}
\newcommand{\bfA}{\mathbf{A}}
\newcommand{\bfB}{\mathbf{B}}
\newcommand{\bfC}{\mathbf{C}}
\newcommand{\bfE}{\mathbf{E}}
\newcommand{\bfJ}{\mathbf{J}}
\newcommand{\bfr}{r}
\newcommand{\bfs}{s}
\newcommand{\bfe}{e}
\newcommand{\abs}[1]{\left\vert #1 \right\vert}
\newcommand{\norm}[1]{\left\Vert #1 \right\Vert}
\newcommand{\ip}[2]{\left< #1, #2 \right>}
\newcommand{\set}[2]{\left\lbrace #1 : #2 \right\rbrace}

\newcommand{\remove}[1]{\ifmmode\text{\sout{\ensuremath{#1}}}\else\sout{#1}\fi}

\title{On $e$-doubly stochastic matrices \\ over hyperbolic (polynomial) systems}

\author[1]{Juyoung Jeong} 
\author[2]{M. Seetharama Gowda}
\author[3]{Sudheer Shukla}

\affil[1]{\small Department of Mathematics\\Soongsil University\\Seoul 06978, South Korea}

\affil[2]{\small Department of Mathematics and Statistics\\University of Maryland Baltimore County\\Baltimore, Maryland 21250, United States}

\affil[3]{\small Department of Mathematics\\University of Maryland, College Park, Maryland 20742, United States}

\date{\today}

\begin{document}
	
\maketitle

\begin{abstract}
    In the setting of a hyperbolic (polynomial) system $(\V, p, e)$ of degree $n$, an $n$-tuple $\bfA = \big[ a_1, a_2, \dots, a_n \big]$ is said to be an $e$-doubly stochastic $\V$-matrix if each $a_i$ belongs to the hyperbolicity cone, has trace $1$, and $a_1 + a_2 + \dots + a_n=e$.
    In this article, we characterize linear preservers of such $\V$-matrices,  describe some connections between $e$-doubly stochasticity and majorization, and study extreme points of the set of all $e$-doubly stochastic $\V$-matrices. We show, for example, that $(i)$ when $n>1$, positive, unital, and trace-preserving transformations are (the only) linear transformations on $\V$ that preserve $e$-doubly stochasticity; $(ii)$ when $\bfA$ is $e$-doubly stochastic, the eigenvalue vector of the linear combination $\sum_{i=1}^{n} r_ia_i$ is majorized by the coefficient vector  $(r_1, r_2, \dots, r_n)^{T} \in \Rn$; and $(iii)$ when $p$ is complete, $e$-doubly stochastic $\V$-matrices induced by (generalized) Jordan frames are extreme points of the compact convex set of all $e$-doubly stochastic $\V$-matrices.
\end{abstract}

\vspace{2ex}



\section{Introduction}

In the classical setting of $\Rn$, a real $n \times n$ matrix $A$ is said to be doubly stochastic if it is  nonnegative and  each row/column of $A$ has sum $1$. Writing $A = \big[ a_1, a_2, \dots, a_n \big]$, where $a_1, a_2, \dots, a_n$ are the columns of $A$, these conditions could be written as
\begin{equation} \label{DS-matrix as e-DS}
    a_i \in \Rn_+,\;\; \tr(a_i)=1 \text{ for all } i = 1, 2, \dots, n, \;\; \text{and} \;\; a_1 + a_2 + \dots + a_n = \one,
\end{equation}
where $\Rn_+$ denotes the nonnegative orthant in $\Rn$, $\one$ denotes the column vector of ones, and $\tr(\bfr)$ denotes the trace of the column vector $\bfr$, namely, the sum of all entries of $\bfr$. If we view the matrix $A$ as a linear transformation on $\Rn$, then the doubly stochasticity property of $A$ reduces to the positive, unital, and trace-preserving properties:
\begin{equation} \label{DS-matrix via operator DS}
    A(\Rn_+) \subseteq \Rn_+,\;\; \, A \one = \one, \;\; \text{and} \;\; \tr(A \bfr) = \tr(\bfr) \text{ for all } \bfr \in \Rn.
\end{equation}
As is well-known, the doubly stochasticity property of a matrix is closely related to (classical) majorization on $\Rn$. Recall that for two given column vectors $\bfr, \bfs \in \Rn$, $\bfr$ is said to be majorized by $\bfs$, symbolically $\bfr \prec \bfs$, if for all $1 \leq k \leq n$, 
\[ \sum_{i=1}^{k} r_i^\downarrow \leq \sum_{i=1}^{k} s_i^\downarrow \;\; \text{and} \;\; \sum_{i=1}^{n} r_i^\downarrow = \sum_{i=1}^{n} s_i^\downarrow, \]
where $\bfr^\downarrow$ denotes the decreasing (non-increasing) rearrangement vector of $\bfr$, etc. In this setting, the results of Hardy-Littlewood-P{\'o}lya and Birkhoff \cite[Theorems II.1.9, II.1.10, II.2.3]{bhatia} assert the following: 
\begin{itemize}
    \item $\bfr \prec \bfs$ if and only if $\bfr$ is in the convex hull of all vectors obtained by permuting the coordinates of $\bfs$;
    \item $\bfr \prec \bfs$ if and only if $\bfr = D \bfs$ for some doubly stochastic matrix $D$;
    \item A real $n \times n$ matrix $A$ is doubly stochastic if and only if $A \bfr \prec \bfr$ for all $\bfr \in \Rn$; 
    \item The set of all doubly stochastic matrices is a compact convex set whose extreme points are permutation matrices.
\end{itemize}

In his study of mixed discriminants, Bapat \cite{bapat} introduces the concept of a doubly stochastic $n$-tuple consisting of $n\times n$ positive semidefinite matrices each with trace $1$ and total sum $I$ (the identity matrix). A similar concept appears in $C^{\ast}$-algebras \cite{alberti-uhlmann-paper, alberti-uhlmann}. Motivated by the concept of $e$-doubly stochastic $n$-tuple introduced by Gurvits \cite{gurvits1}, we study an analogue of \eqref{DS-matrix as e-DS} in hyperbolic (polynomial) systems, focusing on majorization and related issues. 
To set our framework, we first recall essential definitions from hyperbolic systems.

Let $\V$ be a real (nonzero) finite-dimensional vector space over $\R$ and $e$ be a nonzero element of $\V$. We consider a homogeneous polynomial $p$ of degree $n$ (relative to some fixed basis in $\V$) such that $p(e) \neq 0$ and for each $x \in \V$, the univariate polynomial $t \mapsto p(te - x)$ has (only) real zeros. Such a polynomial $p$ is said to be a \emph{hyperbolic polynomial in direction $e$} \cite{garding}. We call the triple $(\V, p, e)$ a \emph{hyperbolic system of degree $n$}. Now, for any $x \in \V$, we define the \emph{eigenvalue vector} $\lambda(x) = \big( \lambda_1(x), \lambda_2(x), \dots, \lambda_n(x) \big)^{T} \in \Rn$, where $\lambda_1(x) \geq \lambda_2(x) \geq \dots \geq \lambda_n(x)$ are the roots of $p(te - x) = 0$. The entries of $\lambda(x)$ are called the \emph{eigenvalues} of $x$, and the \emph{trace} of $x$ is defined as $\tr(x) \coloneq \sum_{i=1}^{n} \lambda_i(x)$. Then the \emph{eigenvalue map} $\lambda : \V \to \Rn$ induces the closed convex cone $\lplus$ with nonempty interior (under any norm on $\V$) \cite{garding, renegar}: 
\[ \lplus \coloneq \set{x \in \V}{\lambda(x) \geq 0}. \] 
It also induces  the following semi-inner product on $\V$ \cite{bauschke et al}:
\[ \ip{x}{y} \coloneq \frac{1}{4} \Big[ \norm{\lambda(x + y)}^2 - \norm{\lambda(x - y)}^2 \Big] \text{ for all } x, y \in \V. \]
As in \cite{bauschke et al}, we say that the hyperbolic system $(\V, p, e)$ is \emph{complete} (or $p$ is complete) if $\lambda(x) = 0$ only when $x = 0$. 

Two primary examples of complete hyperbolic systems are $(\Rn, p, \one)$, where $p(x) = x_1 x_2 \cdots x_n$ (product of the entries of a vector $x \in \Rn$) and $(\Hn, \det, I_{n})$, where $\Hn$ denotes the space of all $n \times n$ (complex) Hermitian matrices, $\det(X)$ is the determinant of $X \in \Hn$, and $I_{n}$ denotes the identity matrix. Going beyond these two examples, we can consider the system $(\V, \det, e)$ induced by a Euclidean Jordan algebra $\V$ of rank $n$ and unit $e$, where the underlying polynomial is the determinant. 

Now consider a hyperbolic system $(\V, p, e)$ of degree $n$. Given an $n$-tuple $(a_1, a_2, \dots, a_n) \in \V^n$, we write $\bfA = \big[ a_1, a_2, \dots, a_n \big]$ and call it a \emph{$\V$-matrix}. Following Gurvits \cite{gurvits1}, we say that $\bfA$ is \emph{$e$-doubly stochastic} if 
\begin{equation} \label{defn of e-DS}
    a_i \in \lplus, \;\; \tr(a_i) = 1 \text{ for all } i=1, 2, \dots, n, \;\; \text{and} \;\; a_1 + a_2 + \dots + a_n = e.
\end{equation}

In our previous paper \cite{gowda-jeong-shukla1}, some preliminary results on such matrices were described. Furthering this study, in the current paper, we raise the following questions and provide some answers:

\begin{itemize}
    \item[(Q1)] \textit{What are the linear preservers of $e$-doubly stochastic matrices?} By way of providing an answer, in Section 4, we consider linear transformations $T$ on $\V$ such that for every $e$-doubly stochastic $\V$-matrix $\bfA = \big[ a_1, a_2, \dots, a_n \big]$, the $\V$-matrix $\bfB = \big[ T(a_1), T(a_2), \dots, T(a_n) \big]$ is $e$-doubly stochastic. When $n > 1$, we characterize these as \emph{operator doubly stochastic transformations} on $\V$ defined by 
    \[ T(\lplus) \subseteq \lplus, \;\; T(e)=e, \;\; \text{and} \;\; \tr(T(x)) = \tr(x). \]
    Next, we consider real square matrices $D = \big[ d_{ij} \big]$ for which $\bfB \coloneq \big[ b_1, b_2, \dots, b_n \big]$, where $b_i = \sum_{j=1}^{n} d_{ij} a_j$, is $e$-doubly stochastic for every $e$-doubly stochastic $\V$-matrix $\bfA = \big[ a_1, a_2, \dots, a_n \big]$. While doubly stochasticity of $D$ is sufficient for this to hold, we show that this condition becomes necessary in the setting of Euclidean Jordan algebras.
    
    \item[(Q2)] \textit{How is $e$-doubly stochasticity related to majorization?} We answer this in Section 5. Given an $e$-doubly stochastic $\V$-matrix $\bfA = \big[ a_1, a_2, \dots, a_n \big]$, a vector $\bfr = (r_1, r_2, \dots, r_n)^{T} \in \Rn$, and an element $x \in \V$, we define $\bfA \bfr \coloneq \sum_{i=1}^{n} r_i a_i$ and $\bfA^{T} x \coloneq \big( \! \ip{a_1}{x}, \ip{a_2}{x}, \dots, \ip{a_n}{x} \! \big)^{T}$, and prove the following majorization results:
    \[ \lambda (\bfA \bfr) \prec \bfr \;\; \text{and} \;\; \bfA^{T} x \prec \lambda(x). \]
    As a consequence, answering a problem posed in \cite{gowda-jeong-shukla1}, we show that if $\bfA$ and $\bfB$ are $e$-doubly stochastic, then the linear transformation $\bfA \odot \bfB \coloneq \sum_{i=1}^{n} a_i \otimes b_i$ is $\lambda$-doubly stochastic, that is,
    \[ \lambda \Big( \sum_{i=1}^{n} \ip{b_i}{x} a_i \Big) \prec \lambda(x) \text{ for all } x \in \V. \]
    This result can be regarded as a generalization of Schur's theorem, which asserts that the diagonal of a Hermitian matrix is majorized by its eigenvalue vector.
    
    \item[(Q3)] \textit{When $p$ is complete, what are the extreme points of the compact convex set of all $e$-doubly stochastic $\V$-matrices?} We provide some partial answers in Section 6. By introducing the concept of a  (generalized) Jordan frame in a hyperbolic system, we show that the $\V$-matrices induced by (generalized) Jordan frames form a subset of the set of all extreme points, and in the setting of the Jordan spin algebra, these are the only extreme points.
\end{itemize}

\section{Preliminaries}

The real Euclidean space $\Rn$ (whose elements are viewed as column vectors) carries the usual inner product. In $\C^n$, for $u = (u_1, u_2, \ldots, u_n)^{T} \in \C^n$, $u^{\ast} = (\overline{u_1}, \overline{u_2}, \ldots, \overline{u_n})$ is the conjugate transpose of $u$, and the inner product is given by $\ip{u}{v} = v^{\ast}u$. For $\bfr, \bfs \in \Rn$, we use the majorization notation $\bfr \prec \bfs$ as defined in the Introduction. Also, we write $\bfr \leq \bfs$ when $\bfs - \bfr \in \Rn_+$.

\subsection{Hyperbolic systems} 

Throughout this paper, we consider a generic hyperbolic system $(\V, p, e)$ of degree $n$ with its hyperbolicity cone $\lplus$; see Section 1.
For $a, x, y \in \V$, we write, as in $\Rn$,
\[ a \geq 0 \text{ when } a \in \lplus \;\; \text{and} \;\; x \leq y  \text{ when } y - x \in \lplus. \]
Regarding the eigenvalue map $\lambda$, we frequently use the property
\[ \lambda(te + x) = t\one + \lambda(x) \text{ for all } t \in \R,\, x \in \V. \]

\begin{definition} 
    We say that the hyperbolic system $(\V, p, e)$ is \emph{complete} \textup{(}or $p$ is complete\textup{)} if the condition $\lambda(x) = 0 \Rightarrow x = 0$ holds.
\end{definition}

We remark that $p$ is complete if and only if $\lplus$ is pointed (in which case, $\lplus$ is a proper/regular cone).

\begin{proposition} \label{bauschke et al} 
    \textup{\cite[Theorem 4.2 and Proposition 4.4]{bauschke et al}} 
    Consider a hyperbolic system $(\V, p, e)$ of degree $n$. Then, corresponding to $\lambda : \V \to \Rn$, 
    \begin{equation} \label{ip}
        \ip{x}{y} \coloneq \frac{1}{4} \Big[ \norm{\lambda(x+y)}^2 - \norm{\lambda(x-y)}^2 \Big] \quad ( x, y \in \V)
    \end{equation}
    defines a bilinear map with $\ip{x}{x} \geq 0$ for all $x$; hence defines a semi-inner product and induces the semi-norm $\norm{x} \coloneq \sqrt{\ip{x}{x}}$. Moreover,
    \begin{equation} \label{semi-ftvn condition}
        \ip{x}{y} \leq \ip{\lambda(x)}{\lambda(y)} \leq \norm{x} \norm{y} \text{ for all } x, y \in \V.
    \end{equation}
    When $p$ is complete, \eqref{ip} defines an inner product.
\end{proposition}

Throughout the paper, for simplicity, we use the same (semi-)inner product notation in $\Rn$ as well as in $\V$. In a hyperbolic system $(\V, p, e)$, relative to the above semi-inner product, we define the \emph{trace} of an element $x$ in $\V$ by
\begin{equation} \label{trace}
    \tr(x) \coloneq \ip{x}{e} = \ip{\lambda(x)}{\one} = \text{sum of all the entries in $\lambda(x)$}. 
\end{equation}

We see that $x \mapsto \tr(x)$ is a linear functional and $\tr(e) = n$. Based on the above semi-inner product, we define the \emph{tensor product} of two elements $a, b \in \V$  as a linear transformation on $\V$:
\[ (a\otimes b)(x) \coloneq \ip{b}{x}a \text{ for all } x \in \V. \]
We define the \emph{dual} of the hyperbolicity cone $\lplus$:
\[ \lplus^{\ast} \coloneq \set{x \in \V}{\ip{x}{y} \geq 0, \,\, \forall\, y \in \lplus}. \]

\begin{proposition} \label{SG theorem} 
    \textup{\cite{gowda-jeong-shukla1}} Consider a hyperbolic system $(\V, p, e)$ and the corresponding semi-inner product as in \eqref{ip}. Then, for all $x, y \in \lplus$, we have $\ip{x}{y} \geq 0$. Thus,
    \[ \lplus \subseteq \lplus^{\ast}. \]
\end{proposition}

\begin{proposition} \label{delta is compact convex}
    Suppose $(\V, p, e)$ is a complete hyperbolic system. Then the set
    \[ \Delta \coloneq \set{x \in \V}{x \geq 0,\, \tr(x) = 1} \]
    is nonempty, compact, and convex in $\V$. When $n > 1$ and $0 \leq b \leq e$ with $\tr(b) \leq n-1$, the set
    \[ \Delta(b) \coloneq \set{x \in \Delta}{0 \leq x \leq e - b} \]
    is also nonempty, compact, and convex in $\V$.
\end{proposition}

\begin{proof}
    Clearly, $\frac{1}{n}e \in \Delta$. As $\lplus$ is convex and the trace functional is linear, $\Delta$ is convex. Moreover, since $p$ is complete (by our assumption), the trace functional is continuous with respect to the norm induced by the inner product \eqref{ip}. Since $\lplus$ is closed, $\Delta$ is also closed. Now the boundedness of $\Delta$ follows from the inequalities
    \[ \norm{x} = \norm{\lambda(x)} \leq \sum_{i=1}^{n} \abs{\lambda_i(x)} = \tr(x) = 1, \]
    where the first equality is due to the norm-preserving property of $\lambda$ (see \eqref{semi-ftvn condition}) and the second equality is due to the nonnegativity of $\lambda(x)$. As $\Delta$ is closed, convex, and bounded, and $\V$ is finite-dimensional, we see that $\Delta$ is compact and convex.
    
    Now, assume the conditions imposed on $n$ and $b$. Clearly, $\Delta(b)$ is convex; it is closed (due to the closedness of $\lplus$). As $n > 1$, we have $\delta \coloneq \tr(e - b) \geq 1$ and so $\frac{1}{\delta}(e-b) \in \Delta(b)\subseteq \Delta$. Thus, $\Delta(b)$ is nonempty, compact, and convex.
\end{proof}

\begin{remark}
    \textit{If $n=1$ and $p$ is complete, then $\dim(\V)=1$ and $\Delta=\{e\}$.} This can be seen by noting that when $n=1$, $\tr \!\big( x - \tr(x) e \big) = 0$ for every $x \in \V$.
\end{remark}

\begin{remark}
    Suppose $\dim(\V) = 1$. It is easy to see that $p$ is complete and $\Delta = \{ \frac{1}{n} e \}$. Conversely, suppose $\Delta = \{ \frac{1}{n} e \}$. Then, for any $a \in \V$, we can find $\varepsilon > 0$ such that $a + \varepsilon e > 0$. With $\delta \coloneq \tr(a + \varepsilon e) > 0$, we have $\frac{1}{\delta} \tr(a + \varepsilon e ) =1$ and so $\frac{1}{\delta}(a + \varepsilon e) \in \Delta = \{\frac{1}{n}e\}$. Consequently, $a$ is a multiple of $e$. Hence, $\dim(\V) = 1$. We conclude: 
    \begin{center} 
        $\dim(\V)=1$ if and only if $\Delta = \{ \frac{1}{n}e \}$.
    \end{center}
\end{remark}

We now recall an important result due to  Gurvits \cite[Proposition 1.2]{gurvits1}. This result allows us to make an eigenvalue statement for two elements in a hyperbolic system by knowing a similar statement for two real symmetric matrices. In the following, for a real symmetric matrix $X$, $\lambda(X)$ denotes its vector of eigenvalues written in decreasing order. 

\begin{theorem} \label{gurvits result1} 
    Consider a hyperbolic system $(\V, p, e)$ of degree $n$. Given $a, b \in \V$, there exist $n \times n$ real symmetric matrices $A$ and $B$ such that
    \[ \lambda(t a + s b) = \lambda(t A + s B) \text{ for all } t, s \in \R. \]
    As a consequence, for any $a, b \in \V$,
    \begin{itemize}
        \item[$(i)$] $\lambda(a + b) \prec \lambda(a) + \lambda(b)$, 
        \item[$(ii)$] $a \leq b$ in $\V$ implies $\lambda(a) \leq \lambda(b)$ in $\Rn$. 
    \end{itemize}
\end{theorem}

Note: Item $(i)$ above extends to any finite set of objects.

\subsection{Euclidean Jordan algebras}

Let $(\V, \circ, \ip{\cdot}{\cdot})$ be a Euclidean Jordan algebra of rank $n$ with unit $e$ \cite{faraut-koranyi, gowda-sznajder-tao}. We write $\Sn$ ($\Hn$) for the Euclidean Jordan algebra of all $n \times n$ real symmetric (respectively, complex Hermitian) matrices. In the algebra $\V$, a nonzero element $c$ is an \emph{idempotent} if $c\circ c=c$; an idempotent is a \emph{primitive idempotent} if it cannot be written as a sum of two other idempotents. By the spectral theorem, every element $x \in \V$ is a linear combination of $n$ mutually orthogonal primitive idempotents. (Here, the coefficients are called the eigenvalues of $x$, and the number of nonzero eigenvalues is the rank of $x$).

The polynomial $p(x) \coloneq \det(x)$, the product of the eigenvalues of $x$, defines a complete hyperbolic polynomial of degree $n$ in the direction $e$. So, $(\V, \det, e)$ becomes the corresponding hyperbolic system in which the inner product (\ref{ip}) reduces to the \emph{trace inner product} $\ip{x}{y} = \tr(x \circ y)$. In this setting, the cone of squares, $\set{x \circ x}{x \in \V}$, is the hyperbolicity cone of $(\V, \det, e)$. We use the same notation in both the algebra and its corresponding hyperbolic system. As in Proposition \ref{delta is compact convex}, consider the compact convex set
\[ \Delta = \set{x \in \V}{x \geq 0, \tr(x) = 1}. \]

It is known \cite[Proposition IV.3.2]{faraut-koranyi} that the set of all extreme points of $\Delta$ is given by
\begin{equation} \label{eq: extreme of primitive idempotents}
    \Ext(\Delta) = \Tau\coloneq\text{the set of all primitive idempotents}.
\end{equation}
Given an idempotent $c \neq e$ with $\rank(c) = k$ and $\gamma \in \{ 0, \frac{1}{2}, 1 \}$, let
\[ \V(c, \gamma) = \set{x \in \V}{x \circ c = \gamma x}. \] 
Then, by the Peirce (orthogonal) decomposition theorem \cite[page 62]{faraut-koranyi}, we have
\[ \V = \V(c, 0) \oplus \V(c, \tfrac{1}{2}) \oplus \V(c, 1). \] 
Moreover, $\V(c, 1)$ and $\V(c, 0)$ are subalgebras of $\V$ of rank $k$ and $n-k$, respectively.

\begin{lemma}\label{EJA lemma}
    In a Euclidean Jordan algebra, consider an idempotent $c \neq e$. Then
    \[ \set{x \in \Delta}{0 \leq x \leq e-c} = \Delta \cap \V(c, 0). \]
\end{lemma}

\begin{proof}
    Let $x \in \Delta$ and $0 \leq x \leq e - c$. Then, $0 \leq \ip{x}{c} \leq \ip{e - c}{c} = 0$; thus $\ip{x}{c} = 0$. As $x \geq 0$ and $c \geq 0$, we have $x \circ c = 0$. So, $x \in \V(c, 0)$. Consequently, $x \in \Delta \cap \V(c, 0)$.
    
    Going the other way, let $y \in \Delta \cap \V(c, 0)$. We only need to show that $y \leq e-c$. Now, for every primitive idempotent $d$ in $\V(c, 0)$, we have $c \circ d = 0$ and so $c + d$ is an idempotent in $\V$; hence $c + d \leq e$. (Note that every idempotent in $\V$ is less than or equal to $e$.) As $y \in \V(c, 0)$ with $y \geq 0$ and $\tr(y) = 1$, thanks to the spectral theorem, we can write $y$ as a convex combination of primitive idempotents, (such as $d$) in $\V(c, 0)$. This means that $y + c \leq e$, or equivalently, $0 \leq y \leq e - c$.
\end{proof}

\subsection{Doubly stochastic transformations} 

Motivated by the results of Hardy-Littlewood-P{\'o}lya and \eqref{DS-matrix via operator DS}, we define several types of doubly stochastic linear transformations in the setting of a hyperbolic system $(\V, p, e)$ of degree $n$. For a linear transformation $T : \V \to \V$, we say that 
\begin{itemize}
    \item[$(a)$] $T$ is \emph{convex-hull doubly stochastic} if $T(x) \in \conv\,[x]$ for all $x\in \V$, where $[x] \coloneq \set{y \in \V}{\lambda(y) = \lambda(x)}$ denotes the $\lambda$-orbit of $x$ and `$\conv$' is an abbreviation for the convex hull; 
    \item[$(b)$] $T$ is \emph{$\lambda$-doubly stochastic} if $\lambda(T(x)) \prec \lambda(x)$ for all $x \in \V$;
    \item[$(c)$] $T$ is \emph{operator doubly stochastic} if 
    \begin{equation} \label{Operator DS}
        T(\lplus) \subseteq \lplus, \;\; T(e)=e, \;\; \text{and} \;\; \tr(T(x)) = \tr(x) \text{ for all } x \in \V.
    \end{equation}
\end{itemize}

The above concepts were introduced and studied in \cite{gowda-jeong-shukla2}. In particular, the following implications were shown (using $\DS$ as an abbreviation for `doubly stochastic'):
\begin{equation} \label{convex to lambda to operator ds}
    \text{convex-hull DS} \implies \text{$\lambda$-DS} \underset{\text{$p$ complete}}{\implies} \text{operator DS}.
\end{equation}

We remark that these concepts coincide in the classical setting of the hyperbolic system $(\Rn, p, \one)$, where $p(x)$ is the product of all entries of the vector $x$. They also coincide in the setting of Euclidean Jordan algebras \cite{gowda-jeong-shukla2}.

\section{\texorpdfstring{$e$}{e}-doubly stochastic \texorpdfstring{$\V$}{V}-matrices}

Consider a hyperbolic system $(\V, p, e)$ of degree $n$ with the corresponding semi-inner product (\ref{ip}).

\begin{definition} \label{v-matrix defn}
    A \emph{$\V$-matrix} is an $n$-tuple $\bfA \coloneq \big[ a_1, a_2, \dots, a_n \big]$, where $a_i \in \V$ for $1 \leq i \leq n$. \textup{(}We
    then say that $a_1, a_2, \dots, a_n$ are the ``columns" of $\bfA$.\textup{)} Given $\bfA$, $x \in \V$, and $\bfr = (r_1, r_2, \dots, r_n)^{T} \in \Rn$, we define
    \begin{itemize}
        \item[] $\bfA \bfr \coloneq r_1 a_1 + r_2 a_2 + \dots + r_n a_n$ and 
        \item[] $\bfA^{T} x \coloneq \big(\! \ip{a_1}{x}, \ip{a_2}{x}, \dots, \ip{a_n}{x} \!\big)^{T}$.
    \end{itemize}
    If $\bfB = \big[ b_1, b_2, \dots, b_n \big]$ is another $\V$-matrix, we define
    \begin{itemize}
        \item[] $\bfA \ast \bfB \coloneq \big[ \! \ip{a_i}{b_j} \! \big]$ and $\bfA \odot \bfB \coloneq \sum_{i=1}^n a_i \otimes b_i$.
    \end{itemize}
    If $D = \big[ d_{ij} \big]$ is an $n \times n$ real matrix, we define
    \begin{itemize}
        \item[] $D \bfA \coloneq \big[ b_1, b_2, \dots, b_n \big]$, where $b_i \coloneq \sum_{j=1}^{n} d_{ij} a_j$ for all $1 \leq i \leq n$.
    \end{itemize}
\end{definition}

In the above,  $\bfA^{T}$ is used merely as a notation. We note that $\bfA \bfr \in \V$ and $\bfA^{T} x \in \Rn$; these two are related by 
\[ \ip{\bfA \bfr}{x} = \ip{\bfr}{\bfA^{T} x}, \]
where the (semi-)inner product on the left (right) is computed in $\V$ (respectively, in $\Rn$). We also note that $\bfA \ast \bfB$ is an $n \times n$ (Gram) matrix. The following are easy to verify: For any $n \times n$ real matrix $D$,
\[ (D \bfA) \ast \bfB = D (\bfA \ast \bfB), \]
and, for any $\bfr \in \Rn$ and $x \in \V$,
\[ (\bfA \ast \bfB) \bfr = \bfA^{T} (\bfB \bfr) \;\; \text{and} \;\; (\bfA \odot \bfB)(x) = \bfA (\bfB^{T}x). \]

\begin{definition} 
    A $\V$-matrix $\bfA = \big [ a_1, a_2, \dots, a_n \big ]$ is said to be \emph{invertible} if there is a $\V$-matrix ${\bf B} = \big[ b_1, b_2, \dots, b_n \big]$ such that $\bfA \ast \bfB = I_{n}$ (the $n\times n$ identity matrix). We say that $\bfB$ is an inverse of $\bfA$. 
\end{definition}

We observe that if $\bfA = \big[ a_1, a_2, \dots, a_n \big]$ is invertible, then $\{a_1, a_2, \dots, a_n\}$ is linearly independent. The converse holds when $p$ is complete. To see this, suppose $\{a_1, a_2, \dots, a_n\}$ is linearly independent. Since \eqref{ip} defines an inner product when $p$ is complete, the Gram matrix $\bfA \ast \bfA$ is invertible. Let $D$ be its inverse. Then 
\[ (D\bfA) \ast \bfA = D (\bfA \ast \bfA) = I_{n} \]
proves that $D\bfA$ is an inverse of $\bfA$.
\\

The following concept was introduced in \cite[Definition 2.4]{gurvits1}. See \cite{gowda-jeong-shukla1} for some properties.

\begin{definition}
    In a hyperbolic system $(\V,p,e)$ of degree $n$, a $\V$-matrix $\bfA = \big[ a_1, a_2, \dots, a_n \big]$ is said to be \emph{$e$-doubly stochastic} if \eqref{defn of e-DS} holds, i.e.,
    \[ a_i \in \lplus, \;\; \tr(a_i) = 1 \text{ for } i = 1, 2, \dots, n, \;\; \text{and} \;\; a_1 + a_2 + \dots + a_n = e. \]
\end{definition}

\begin{remark}
    If $\{a_1, a_2, \dots, a_k\}$ is a set of trace-one elements with sum $e$, then $k = n$. This can be seen by the linearity of the trace functional:
    \[ n = \tr(e) = \sum_{i=1}^{k} \tr(a_i)=k. \]
    This means that in the above definition of $e$-doubly stochasticity, we cannot replace $n$ (the degree of $p$) by another natural number.
    
    Note that if  $\bfA = \big[ a_1, a_2, \dots, a_n \big]$ is $e$-doubly stochastic, then any permutation of the columns of $\bfA$ yields another $e$-doubly stochastic $\V$-matrix.
\end{remark}

In the following, we describe several ways of constructing $e$-doubly stochastic $\V$-matrices.

\begin{example}\label{n=1 case}
    In any hyperbolic system $(\V, p, e)$ of degree $n$, 
    \[ \mathbf{E} = \Big[\, \frac{1}{n}e, \frac{1}{n}e, \dots, \frac{1}{n}e \,\Big] \] 
    is $e$-doubly stochastic. If $n = 1$, then $\bfA = [\,e\,]$ is the only $e$-doubly stochastic $\V$-matrix. 
\end{example}

\begin{example} \label{example-average}
    Consider a hyperbolic system $(\V,p,e)$ of degree $n$, where 
    \[ \dim(\V) > 1 \text{ and } n > 1.\] 
    As the trace functional is linear and $\tr(e) \neq 0$, the subspace $\set{x \in \V}{\tr(x) = 0}$ has co-dimension one; let $0 \neq a \in \set{x \in \V}{\tr(x) = 0}$ and choose $\varepsilon > 0$ so that
    \[ \lambda \Big( \frac{1}{n}e \pm \varepsilon a \Big) = \frac{1}{n} \one + \varepsilon \lambda(\pm a) \geq 0. \]
    Then $\bfA = \big[ \tfrac{1}{n}e + \varepsilon a, \tfrac{1}{n}e - \varepsilon a, \frac{1}{n}e , \dots, \tfrac{1}{n}e \,\big]$ is $e$-doubly stochastic and has at least two distinct entries.
\end{example}

\begin{example} \label{example-SG} 
    Consider a hyperbolic system $(\V, p, e)$ of degree $n$, where $n > 1$. Let $a \geq 0$ with $\tr(a) = 1$. Define
    \[ a_1 \coloneq a \;\; \text{and} \;\; a_k \coloneq \frac{1}{n-1}(e-a) \text{ for } k = 2, 3, \dots, n. \]
    Then $\bfA \coloneq \big[ a_1, a_2, \dots, a_n \big]$ is $e$-doubly stochastic. This is seen as follows. Since $a \geq 0$ and $\tr(a) = 1$, we have $0 \leq \lambda_i(a) \leq 1$ and $\lambda_i(e - a) = 1 - \lambda_{n-i+1}(a) \geq 0$ for all $i$. This implies that $e - a \geq 0$. As $\tr(e-a)=n-1$ and $n > 1$,  we see that $e - a \neq 0$. Clearly, $a_i \geq 0$, $\tr(a_k) = 1$ for all $k$, and $a_1 + a_2 + \dots + a_n = e$. We conclude that $\bfA$ is $e$-doubly stochastic. 
\end{example}

\begin{proposition} \label{prop:construction of e-DS}
    Consider a hyperbolic system $(\V, p, e)$ of degree $n$, where $p$ is complete and $n > 1$. Let $\bfA = \big[ a_1, a_2, \dots, a_n \big]$ be $e$-doubly stochastic. For $1 \leq k \leq n-1$, define
    \[ \Delta_k \coloneq \Delta(a_1 + a_2 + \dots + a_k) = \set{x \in \Delta}{0 \leq x \leq e - (a_1 + a_2 + \dots + a_k)}. \]
    Then,
    \begin{itemize}
        \item[$(i)$] $a_{k+1} \in \Delta_{k}$ and $\Delta_k$ is a compact convex set;
        \item[$(ii)$] $\Delta \supseteq \Delta_1 \supseteq \dots \supseteq \Delta_{n-1}$;
        \item[$(iii)$] $\Delta_{n-1}$ is a singleton set.
    \end{itemize}
\end{proposition}

\begin{proof} 
    $(i)$ We fix an index $k$ with $1 \leq k \leq n-1$. Then $a_1 + a_2 + \dots + a_k + a_{k+1} \leq a_1 + a_2 + \dots + a_n = e$. Hence,
    \begin{equation} \label{useful inequality}
        0 \leq a_{k+1} \leq e - (a_1 + a_2 + \dots + a_k) \text{ for } 1 \leq k \leq n-1.
    \end{equation}
    This shows that $a_{k+1} \in \Delta_k$. Moreover, $\tr(a_1 + a_2 + \dots + a_k) = k \leq n-1$. Hence, by Proposition \ref{delta is compact convex}, $\Delta_k$ is compact and convex.
    
    $(ii)$ This is easy to verify.
    
    $(iii)$ We claim that if $x\in \Delta_{n-1}$, then $x = e -(a_1 + a_2 + \dots + a_{n-1})$. Suppose $x \in \Delta_{n-1}$ so that $x \in \Delta$ and $0 \leq x \leq e - (a_1 + a_2 + \dots + a_{n-1})$. Then $0 \leq a_1 + a_2 + \dots + a_{n-1} + x \leq e$ and $\tr \!\big( e - (a_1 + a_2 + \dots + a_{n-1} + x) \big) = n - n = 0$; hence, by the completeness of $p$, $a_1 + a_2 + \dots + a_{n-1}+x = e$, i.e., $x = e -(a_1 + a_2 + \dots + a_{n-1})$. 
\end{proof}

The above proposition motivates the following construction. 

\begin{example} \label{example via delta sets}
    Consider a hyperbolic system $(\V,p,e)$ of degree $n$, where $p$ is complete. By Proposition \ref{delta is compact convex}, $\Delta$ is nonempty, compact, and convex. When $n=1$, we have $\Delta=\{e\}$; in this case, $\bfA=[\, e \,]$ is $e$-doubly stochastic.
    
    Now let $n > 1$. We start with any  $a_1 \in \Delta$ and recursively construct $a_2, a_3, \dots, a_n$ as follows: For $1 \leq k \leq n-1$, suppose we have $a_1, a_2, \dots, a_{k} \in \Delta$ with 
    \[ \Delta_{k} \coloneq \Delta(a_1 + a_2 + \dots + a_k) = \set{x \in \Delta}{0 \leq x \leq e - (a_1 + a_2 + \dots + a_k)}. \]
    Since $\tr \!\big( e - (a_1 + a_2 + \dots + a_k) \big) = n-k \leq n-1$, by Proposition \ref{delta is compact convex}, $\Delta_k$ is nonempty (as well as compact and convex). We then choose (any) $a_{k+1} \in \Delta_{k}$. This process leads to the elements $a_1,a_2,\ldots, a_n$, where $a_n\in \Delta_{n-1}$. As $\Delta_{n-1}$ is a singleton set, see the proof of the previous proposition, we verify that
    $\bfA = \big[ a_1, a_2, \dots, a_n \big]$ is an $e$-doubly stochastic $\V$-matrix with $a_{k+1} \in \Delta_{k}$ for $1 \leq k \leq n-1$.
     
    {\it This example, together with Proposition \ref{prop:construction of e-DS}, shows that every $e$-doubly stochastic $\V$-matrix arises in this way.}
\end{example}

Our next example (obtained by specializing $a_k$ in the above construction) will be useful in describing the extreme points of the set of all $e$-doubly stochastic $\V$-matrices; see Theorem \ref{thm:doubly stochastic transformations} below. In what follows, $\Ext(S)$ denotes the set of all extreme points of a convex set $S$. We recall the well-known result of Minkowski; see \cite[Theorem 17.2]{rockafellar}: \textit{In a finite-dimensional space, every nonempty compact convex set is the convex hull of its extreme points.}

\begin{example} \label{example via extreme points}
    Consider a hyperbolic system $(\V,p,e)$ of degree $n$, where $p$ is complete. By Proposition \ref{delta is compact convex}, $\Delta$
    is nonempty, compact, and convex. Then, $\Ext(\Delta)$ is nonempty. When $n = 1$, $\Delta = \{e\}$, so $e \in \Ext(\Delta)$ and $\bfA = [\, e \,]$ is $e$-doubly stochastic. When $n > 1$, we start with (any) $a_1 \in \Ext(\Delta)$ and, for $1 \leq k \leq n-1$, we recursively construct $\Delta_k$ as in the previous example, and choose $a_{k+1} \in \Ext(\Delta_{k})$. Then $\bfA = \big[ a_1, a_2, \dots, a_n \big]$ is $e$-doubly stochastic.
\end{example}

Motivated by the above specialized construction, we introduce the following definition.

\begin{definition} 
    Consider a complete hyperbolic system $(\V, p, e)$ of degree $n$. A $\V$-matrix $\bfA = \big[ a_1, a_2, \dots, a_n \big]$ is said to be a \emph{generalized Jordan frame matrix} \textup{(}in short, a \emph{$\GJF$-matrix}\textup{)}, if there is a permutation $\sigma$ of $\{1, 2, \dots, n\}$ such that the $\V$-matrix $\big[ a_{\sigma(1)}, a_{\sigma(2)}, \dots, a_{\sigma(n)} \big]$ comes from the construction in Example \ref{example via extreme points}.
\end{definition}

In order to explain the appearance/usage of the term `Jordan frame,' we recall the following from \cite{gowda-jeong-shukla1}.

\begin{definition} \label{def:Jordan_frame} 
    Consider a hyperbolic system $(\V, p, e)$ of degree $n$. A set $\{c_1, c_2, \dots, c_n\}$ in $\V$  is said to be a \emph{Jordan frame} if $\lambda(c_i) = (1, 0, \dots, 0)^{T}$ for all $1 \leq i \leq n$ and $c_1 + c_2 + \dots + c_n = e$. The corresponding $\V$-matrix ${\bf J} = \big[ c_1, c_2, \dots, c_n \big]$ is called a \emph{Jordan frame matrix} \textup{(}in short, a \emph{$\JF$-matrix}\textup{)}.
\end{definition}

Regarding Jordan frames, we recall the following. 

\begin{proposition} \label{orthonormality of JF} \textup{\cite[Theorem 4.20]{gowda-jeong-shukla1}}
    Consider a hyperbolic system of degree $n$. Then, every Jordan frame is orthonormal with respect to the semi-inner product defined in \eqref{ip}, hence linearly independent. Consequently, if $\V$ has a Jordan frame, then $\dim(\V) \geq n$.
\end{proposition}

\begin{proposition} 
    In a complete hyperbolic system, every $\JF$-matrix is a $\GJF$-matrix.
\end{proposition}

\begin{proof}
    Suppose $\{c_1, c_2, \dots, c_n\}$ is a Jordan frame in $(\V, p, e)$ and $p$ is complete. We show that $\bfJ = \big[ c_1, c_2, \dots, c_n \big ]$ is a $\GJF$-matrix. As this is obvious when $n = 1$, we assume $n > 1$. Now, $c_1 \geq 0$ and $\tr(c_1)=1$; so $c_1$ is in the compact convex set $\Delta$. Since $c_1$ has rank one (meaning it has only one nonzero eigenvalue), it is an extreme direction of $\lplus$; see \cite[Proposition 3.8]{gowda-jeong-shukla1}. Thus, $c_1 \in \Ext(\Delta)$. Using the definition of a Jordan frame, we verify that $0\leq c_2\leq e-c_1$; hence, $c_2\in \Delta(c_1)$. Now $c_2$ has rank one and is an extreme point of $\Delta$; hence $c_2\in \Ext(\Delta(c_1)).$ Continuing this way, we see that for $2 \leq k \leq n-1$, $c_k \in \Ext(\Delta(c_1 + c_2 + \dots + c_{k-1}))$. Thus $\bfJ$ is a $\GJF$-matrix. In particular, $\bfJ$ is $e$-doubly stochastic. 
\end{proof}

\begin{example}
    Consider a Euclidean Jordan algebra $(\V, \circ, \ip{\cdot}{\cdot})$ with unit $e$ and rank $n$, equipped with the trace inner product. Let $\{c_1, c_2, \dots, c_n\}$ be a Jordan frame in this algebra. This means that each $c_i$ is a primitive idempotent, $\ip{c_i}{c_j} = 0$ for all $i\neq j$, and $c_1 + c_2 + \dots + c_n = e$. As each $c_i \geq 0$ and $\tr(c_i) = 1$, $\{c_1, c_2, \dots, c_n\}$ becomes a Jordan frame in the complete hyperbolic system $(\V, \det, e)$. It follows that the $\JF$-matrix $\bfJ \coloneq \big[ c_1, c_2, \dots, c_n \big]$ is $e$-doubly stochastic. In conclusion, \textit{every Jordan frame in a Euclidean Jordan algebra defines a $\JF$-matrix in the corresponding hyperbolic system.}
\end{example}

The following result characterizes $\GJF$-matrices in the setting of a Euclidean Jordan algebra.

\begin{proposition} \label{In EJA, JF=GJF}
    Consider a Euclidean Jordan algebra $(\V, \circ, \ip{\cdot}{\cdot})$ with unit $e$ and rank $n$ carrying the trace inner product. Then, in the (complete) hyperbolic system $(\V, \det, e)$, every $\GJF$-matrix arises from a Jordan frame of the algebra.
\end{proposition}

\begin{proof} 
    Consider a $\GJF$-matrix $\bfA = \big[ a_1, a_2, \dots, a_n \big]$; without loss of generality, let $\bfA$ come from the construction in Example \ref{example via extreme points}. We show that this must be of the form 
    $\bfJ = \big[ c_1, c_2, \dots, c_n \big]$ for some Jordan frame $\{c_1, c_2, \dots, c_n\}$. To see this, we recall from \eqref{eq: extreme of primitive idempotents} that
    \[ \Ext(\Delta) = \Tau\,(=\text{the set of all primitive idempotents}), \]
    where $\Delta = \set{x \in \V}{x \geq 0,\, \tr(x) = 1}$. Now, since $a_1 \in \Ext(\Delta)$, $a_1$ must be a primitive idempotent, say $e_1\in \Tau$. Then
    \[ \Delta_1 = \Delta(a_1) = \Delta(e_1) = \set{x \in \Delta}{0 \leq x \leq e - e_1}. \]
    
    Now, by Lemma \ref{EJA lemma}, $\Delta(e_1) = \Delta \cap \V(e_1,0) = \set{y \in \V(e_1, 0)}{y \geq 0,\, \tr(y) = 1}.$
    
    Consequently, $\Ext(\Delta(e_1))$ consists of all primitive idempotents in $\V(e_1, 0)$. As primitive idempotents of $\V(e_1,0)$ are primitive idempotents of $\V$,  choosing an element in $\Ext(\Delta(e_1))$ amounts to choosing $e_2 \in \Tau$ with 
    $e_2\circ e_1 = 0$, i.e., with $\ip{e_2}{e_1} = 0$. Thus, $a_2$ must be some $e_2 \in \Tau$. At the next step, $a_3 \in \Ext(\Delta(e_1 + e_2))$ must be a primitive idempotent in $\V(e_1 + e_2, 0)$, so must be of the form $e_3 \in \Tau$ with $\ip{e_3}{e_1 + e_2} = 0$. Since $\ip{e_3}{e_1} \geq 0$ and $\ip{e_3}{e_2} \geq 0$, we see that $e_3$ is orthogonal to both $e_1$ and $e_2$. Continuing this way, we see that $\bfA = \big[ e_1, e_2, \dots, e_n \big]$, where each $e_i \in \Tau$, $\ip{e_i}{e_j} = 0$ for all $i \neq j$. As the sum of $e_i$'s is $e$, the set $\{e_1, e_2, \dots, e_n\}$ is a Jordan frame in $\V$. Thus, $\bfA$ comes from a Jordan frame, that is, $\bfA$ is a $\JF$-matrix. 
\end{proof}

Now, deviating from Euclidean Jordan algebras, we construct the following example.

\begin{example} \label{GJF example 1} 
    On $\R^3$, for $x = (x_1, x_2, x_3)^{T}$, let $p(x) \coloneq x_1^2x_2x_3$ and $e = (1, 1, 1)^{T}$. Then, $(\R^3, p, e)$ is a complete hyperbolic system of degree $4$ with $\lambda(x)=[(x_1, x_1, x_2, x_3)^{T}]^\downarrow \in \R^4$. Since $\lplus=\R^3_+$ and $\tr(x) = 2x_1 + x_2 + x_3$, we have
    \[ \Delta = \set{(x_1, x_2, x_3)^{T} \in \R^3}{x_1, x_2, x_3 \geq 0,\; 2x_1 + x_2 + x_3 = 1}. \]
    Then $\Ext(\Delta) = \big\{ (\frac{1}{2}, 0, 0)^{T}, (0, 1, 0)^{T}, (0, 0, 1)^{T} \big\}$. Now, starting from $(0,0,1)^{T}$, as in Example \ref{example via extreme points}, we construct the $\mathrm{GJF}$-matrix $\bfA=\big[ a_1, a_2, a_3, a_4 \big]$, where $a_1 = (0, 0, 1)^{T}$, $a_2 = a_3 =(\frac{1}{2}, 0, 0)^{T}$, and $a_4 = (0, 1, 0)^{T}$. However, as $a_1$ and $a_4$ are the only rank-one elements with trace one, we note that there is no $\JF$-matrix in this system.
\end{example}

\section{Linear preservers of \texorpdfstring{$e$}{e}-doubly stochasticity}

We now study the problem of describing linear preservers of $e$-doubly stochastic $\V$-matrices. We consider two types: $(i)$ A linear transformation $T : \V \to \V$ such that $T(\bfA) \coloneq \big[ T(a_1), T(a_2), \dots, T(a_n) \big]$ is $e$-doubly stochastic whenever $\bfA= \big[ a_1, a_2, \dots, a_n \big]$ is $e$-doubly stochastic, and $(ii)$ a matrix $D$ such that $D\bfA$ is $e$-doubly stochastic for every $e$-doubly stochastic $\bfA$.
\\

The following result characterizes transformations of the first type.

\begin{theorem} \label{linear preservers}
    Consider a hyperbolic system $(\V, p, e)$ of degree $n > 1$. Then a linear transformation $T: \V \to \V$ preserves $e$-doubly stochastic $\V$-matrices if and only if $T$ is operator doubly stochastic, that is,
    \[ T(\lplus)\subseteq \lplus, \;\; T(e) = e, \;\; \text{and} \;\; \tr(T(x)) = \tr(x) \text{ for all } x \in \V. \] 
\end{theorem}

\begin{proof} 
    Suppose that $T$ is operator doubly stochastic. It is easy to see that if $\big[ a_1, a_2, \dots, a_n \big]$ is $e$-doubly stochastic, then $\big[ T(a_1), T(a_2), \dots, T(a_n) \big]$ is $e$-doubly stochastic. 
    
    To see the `only if' part, assume that $\big[ T(a_1), T(a_2), \dots, T(a_n) \big]$ is $e$-doubly stochastic for every $e$-doubly stochastic $\bfA=[a_1,a_2,\dots, a_n]$. Taking $\bfA = \big[ \tfrac{1}{n}e, \tfrac{1}{n}e, \dots, \tfrac{1}{n}e \big]$, we see that $T(e) = e$. When $\dim(\V)=1$, any element $x$ of $\V$ is a scalar multiple of $e$. In this case, it is easy to see that $T(\lplus)\subseteq \lplus$ and $ \tr(T(x)) = \tr(x)$ for all $x \in \V$.
    
    Now assume that $\dim(\V)>1$. Consider any $a \in \V$ with $\tr(a) = 0$. Then, as in Example \ref{example-average} (recall our assumption that $n>1$), there exists an $\varepsilon > 0$ such that 
    \[ \Big[\, \frac{1}{n}e + \varepsilon a, \frac{1}{n}e - \varepsilon a, \frac{1}{n} e, \dots, \frac{1}{n}e \,\Big] \] 
    is $e$-doubly stochastic. Then 
    \[ \Big[\, \frac{1}{n}T(e) + \varepsilon T(a), \frac{1}{n}T(e) - \varepsilon T(a), \frac{1}{n} T(e), \dots, \frac{1}{n}T(e) \,\Big] \] 
    is $e$-doubly stochastic as well. Consequently, $1 = \tr \big( \tfrac{1}{n} T(e) + \varepsilon T(a) \big)$; using the linearity of the trace functional and $T(e) = e$, we see that $\tr(T(a)) = 0$. This shows that $\tr(T(a)) = 0$ for any $a \in \V$ with $\tr(a) = 0$. Now, for any $x \in \V$, we let $a \coloneq x - \frac{1}{n}\tr(x)e$ so that $\tr(a) = 0$ and (consequently) $\tr(T(a)) = 0$. This implies
    \[ \tr(T(x))=\tr(x) \text{ for all } x \in \V. \]
    
    Lastly, to show that $T(\lplus)\subseteq \lplus$, we proceed as follows: Consider $0 \neq a \geq 0$.
    First suppose $\tr(a) > 0$. By scaling, we may assume without loss of generality that $\tr(a) = 1$. Then, as in Example \ref{example-SG} (recall $n > 1$), we construct an $e$-doubly stochastic $\V$-matrix with first column $a$. Then, by the imposed condition on $T$, $T(a) \geq 0$. If $\tr(a)=0$, then for any positive $\varepsilon$, $a + \varepsilon e \geq 0$ and $\frac{1}{n\varepsilon}\tr(a + \varepsilon e) = 1$. Then $T(a + \varepsilon e) \geq 0$. Letting $\varepsilon \to 0$, we see that $T(a)\geq 0$.
    Thus, $T$ is operator doubly stochastic. 
\end{proof}

\begin{remark}
    The above result may fail when $n = 1$. Here is an example. On $\V = \R^2$, for $x=(x_1, x_2)^{T}$, let $p(x) = x_1$ and $e = (1, 1)^{T}$. In this setting, $\lplus = \set{x}{x_1 \geq 0}$ and $\bfA = [\, e \,]$ is the only $e$-doubly stochastic $\V$-matrix. Then the transformation $T : \V \to \V$ defined by $T(x) =(2x_1 - x_2, x_1)^T$ takes $e$ to $e$ (so preserves $e$-doubly stochasticity), but is not operator doubly stochastic.
\end{remark}

Regarding transformations of the second type, we have the following.

\begin{theorem}
    Consider a hyperbolic system $(\V,p,e)$, where $n>1$. If a real $n \times n$ matrix $D$ is doubly stochastic, then for every $e$-doubly stochastic $\V$-matrix $\bfA$, $D\bfA$ is $e$-doubly stochastic. The converse holds when $\dim(\V) > 1$ and $D$ is a nonnegative matrix.  In particular, the converse holds when $\V$ has a Jordan frame.
\end{theorem}

\begin{proof} 
    Suppose $D$ is a doubly stochastic matrix and $\bfA$ is an $e$-doubly stochastic $\V$-matrix. Using the linearity of the trace functional, it is easy to show that $D\bfA$ is $e$-doubly stochastic. 
    
    We now prove the converse. Suppose $\dim(\V) > 1$ and $D$ is a nonnegative matrix. Assume that $D\bfA$ is $e$-doubly stochastic for every $\bfA$ that is $e$-doubly stochastic.
    
    Taking $\bfA = \big[ \frac{1}{n}e, \frac{1}{n}e, \dots, \frac{1}{n}e \big]$ and writing $D\bfA = \big[ b_1, b_2, \dots, b_n \big]$, we observe that $\tr(b_i)=1$. This implies that $\sum_{j=1}^{n}d_{ij}=1$ for all $i$. Hence, each row of $D$ has sum $1$.
    
    Since $\dim(\V) > 1$, we can choose $0 \neq a \in \V$ such that $\tr(a) = 0$. Let $\bfA = \big[ \frac{1}{n}e + \varepsilon a, \frac{1}{n}e - \varepsilon a, \frac{1}{n}e , \dots, \frac{e}{n}\,\big ]$, see Example \ref{example-average}, and write $D\bfA = \big[ b_1, b_2, \dots, b_n \big]$. Then, for all $i$, $b_i = \frac{1}{n} e + \varepsilon(d_{i1} - d_{i2}) a$. Since $\sum_{i=1}^{n}b_i=e$ and $a \neq 0$, we get 
    \[ \sum_{i=1}^{n} d_{i1} = \sum_{i=1}^{n} d_{i2}. \] 
    This means that the first two columns of $D$ have the same column sum. By considering $e$-doubly stochastic $\V$-matrices obtained by permuting the columns of $\bfA$ and mimicking the above argument, we see that all columns of $D$ have the same column sum. Since the sum of all entries of $D$ equals $n$ (recall that each row of $D$ has sum $1$), we see that each column of $D$ has sum $1$. As $D$ is assumed to be nonnegative, it is doubly stochastic. 
    
    Now, suppose that $\V$ has a Jordan frame, say, $\{c_1, c_2, \dots, c_n\}$ and consider the corresponding $e$-doubly stochastic matrix $\mathbf{J} = \big[c_1, c_2, \dots, c_n \big]$. We verify that $\dim(\V) > 1$ and $D$ is a nonnegative matrix. Now, from Proposition \ref{orthonormality of JF}, $\mathbf{J} \ast \mathbf{J}= I_{n}$ and  $\dim(\V)\geq n > 1$. As $D{\bf J}=[b_1,b_2,\dots, b_n]$ is $e$-doubly stochastic, we have 
    \[ 0\leq \ip{b_i}{c_j} = \sum_{k=1}^{n} d_{ik} \ip{c_k}{c_j} = \sum_{k=1}^{n} d_{ik} \delta_{kj} = d_{ij} \text{ for all } i, j. \]
    Thus, $D$ is a nonnegative matrix. Hence, by what we have proved earlier, $D$ is doubly stochastic.
\end{proof}

\begin{remark}
    The converse part of the above result may fail when   $\dim(\V) = 1$. For example, consider $\V = \R$, $p(x) = x^2$, and $e = 1$ so that $\Delta = \{ \frac{1}{2} \}$ and $\tr(x) = 2x$. In this system, the only $e$-doubly stochastic matrix is $\bfA = \big[ \frac{1}{2}, \frac{1}{2} \big]$. Then the matrix $D = \begin{bmatrix} 1 & 0 \\ 1 & 0 \end{bmatrix}$ is not doubly stochastic, although $D\bfA\, ( = \bfA$) is $e$-doubly stochastic.
\end{remark}

\section{Majorization results}

In the setting of $\Rn$, for an $n \times n$ doubly stochastic matrix $A$, we know that $A r \prec r$ and $A^{T} r \prec r$ for all $r \in \Rn$. In what follows, we state analogs of these in hyperbolic systems. First, we recall the following result from \cite[Proposition 5.6]{gowda-jeong-shukla1}.

\begin{proposition} 
    Suppose $\bfA = \big[ a_1, a_2, \dots, a_n \big]$ is $e$-doubly stochastic. Then
    \begin{equation} \label{lambda sum majorized by onek}
        \lambda(a_1 + \dots + a_k) \prec \one_k \text{ for all } 1 \leq k \leq n,
    \end{equation}
    where $\one_k \coloneq (1, 1, \dots, 1, 0, \dots, 0)^{T}$ has $k$ number of ones.
\end{proposition}

\begin{theorem} \label{majorization for v-matrices1}
    If $\bfA = \big[ a_1, a_2, \dots, a_n \big]$ is $e$-doubly stochastic, then
    \[ \lambda(\bfA \bfr) \prec \bfr \text{ for all } \bfr \in \Rn. \]
    The converse holds when $p$ is complete.
\end{theorem}

\begin{proof}
    First suppose that $\bfA = \big[ a_1, a_2, \dots, a_n \big]$ is $e$-doubly stochastic. Let $\bfr = (r_1, r_2, \dots, r_n)^{T} \in \Rn$ and $x= r_1 a_1 + r_2 a_2 + \dots + r_n a_n$. Since $e = a_1 + a_2 + \dots + a_n$, we have $x + te =(r_1 + t)a_1 + (r_2 + t)a_2 + \dots + (r_n + t)a_n$ for any $t \in \R$. Also, for any $\bfs \in \Rn$, note that $\bfr \prec \bfs$ if and only if $\bfr + t \one \prec \bfs + t \one$. Hence, by adding a suitable multiple of $e$ to $x$ and rearranging the indices if necessary, we may assume that $r_1 \geq r_2 \geq \dots \geq r_n = 0$.
    
    Now, define $b_k = a_1 + a_2 + \dots + a_k$ for $k = 1, 2, \dots, n$. Then
    \[ \sum_{i=1}^{n} r_i a_i = \sum_{k=1}^{n-1} (r_k - r_{k+1}) b_k. \]
    Note that $r_k - r_{k+1} \geq 0$ for all $k = 1, 2, \dots, n-1$. From \eqref{lambda sum majorized by onek}, $\lambda(b_k) \prec \one_k$ for all $k$; so, 
    \begin{align*}
        \lambda \bigg( \sum_{i=1}^{n} r_i a_i \bigg) 
        &= \lambda \bigg( \sum_{k=1}^{n-1} (r_k - r_{k+1}) b_k \bigg) \\
        &\prec \sum_{k=1}^{n-1} (r_k - r_{k+1}) \lambda(b_k) 
        \prec \sum_{k=1}^{n-1} (r_k - r_{k+1}) \one_k,
    \end{align*}
    where the first majorization inequality is obtained by repeatedly invoking Item $(i)$ in Theorem \ref{gurvits result1}. Furthermore, we easily verify that 
    \[ \sum_{k=1}^{n-1} (r_k - r_{k+1}) \one_k= (r_1, r_2, \dots, r_n)^{T} = r. \]
    This proves that $\lambda(\bfA \bfr) \prec \bfr$.
    
    For the converse, assume that $\lambda(\bfA \bfr) \prec \bfr$ for all $\bfr \in \Rn$. Now, for any index $i$, taking $\bfr = \bfe_i \in \Rn$, where $\bfe_i$ is the $i$th coordinate vector, we see that $\lambda(a_i) \prec \bfe_i$. This means that $\lambda(a_i) \geq 0$, i.e., $a_i \in \lplus$. Also, $\tr(a_i) = 1$. Now, taking $\bfr = \one$, we see $\lambda(a_1 + a_2 + \dots + a_n) \prec \one$, hence $\lambda(a_1 + a_2 + \dots + a_n) = \one$. By the completeness of $p$, we get $a_1 + a_2 + \dots + a_n = e$.
\end{proof}

\begin{remark}
    When $\bfA$ comes from a Jordan frame, the majorization inequality in the above theorem becomes $\lambda(\bfA \bfr) = \bfr^\downarrow$ for all $\bfr \in \Rn$; see \cite[Theorem 4.16]{gowda-jeong-shukla1}. 
\end{remark}

In the converse part of our next result, we impose a condition on the dual of the hyperbolicity cone. We recall from Theorem \ref{SG theorem} that $\lplus\subseteq \lplus^\ast.$
\begin{theorem} \label{majorization for v-matrices2} 
    If $\bfA = \big[ a_1, a_2, \dots, a_n \big]$ is $e$-doubly stochastic, then
    \[ \bfA^{T} x \prec \lambda (x) \text{ for all } x \in \V. \] 
    The converse holds when $p$ is complete and $\lplus = \lplus^{\ast}$.  In particular, the converse holds in the setting of a Euclidean Jordan algebra.
\end{theorem}

\begin{proof} 
    We fix $x\in \V$. Then, by definition, $\bfA^{T} x = (\ip{a_1}{x}, \ip{a_2}{x}, \dots, \ip{a_n}{x} )^{T}$. By rearranging $a_i$'s, if necessary, we can assume that $\ip{a_1}{x} \geq \ip{a_2}{x} \geq \dots \geq \ip{a_n}{x}$. From \eqref{lambda sum majorized by onek}, $\lambda(a_1 + \dots + a_k) \prec \one_k$ for $1 \leq k \leq n$. Then for $1 \leq k \leq n$, we have
    \begin{align*}
        \ip{a_1}{x} + \ip{a_2}{x} + \dots + \ip{a_k}{x}
        &= \ip{a_1 + a_2 + \dots + a_k}{x} \\
        &\leq \ip{\lambda(a_1 + a_2 + \dots + a_k)}{\lambda(x)} \\
        &\leq \ip{\one_k}{\lambda(x)} \\
        & = \lambda_1(x) + \lambda_2(x) + \dots + \lambda_k(x),
    \end{align*}
    where the first inequality comes from the inner product expanding property of $\lambda$, see \eqref{semi-ftvn condition}, and the second inequality comes from \eqref{lambda sum majorized by onek} and \cite[Problem II.5.16]{bhatia}. Furthermore, when $k = n$,
    \begin{align*}
        \ip{a_1}{x} + \ip{a_2}{x} + \dots + \ip{a_n}{x}
        &= \ip{a_1 + a_2 + \dots + a_n}{x} \\
        & = \ip{e}{x} \\
        & = \lambda_1(x) + \lambda_2(x) + \dots + \lambda_n(x).
    \end{align*}
    Thus, $\bfA^{T} x \prec \lambda(x)$.
    
    To see the converse, assume that $p$ is complete and $\lplus = \lplus^{\ast}$. Let $x \in \lplus$ so that $\lambda(x) \in \Rnp$. Since $\bfA^{T} x \prec \lambda(x)$, $\bfA^{T} x \in \Rnp$ as well. So, for each $x\in \lplus$ and every index $i$, $\ip{a_i}{x} \geq 0$. As $\lplus=\lplus^{\ast}$, we have $a_i\in \lplus$ for all $i$.
    
    Next, we have $\bfA^{T} e \prec \lambda(e) = \one$, implying $\bfA^{T} e = \one$. Thus $\tr(a_i) = \ip{a_i}{e} = 1$ for each $i$. Lastly, writing $a = a_1 + a_2 + \dots + a_n$, $\bfA^{T} x \prec \lambda(x)$ implies that $\ip{a}{x} = \ip{e}{x}$ for all $x \in \V$. As we are working with an inner product (recall that $p$ is complete), we obtain $a = e$.\\
    Finally, in the setting of a Euclidean Jordan algebra, $p$ (the determinant) is complete, and the hyperbolicity cone (the symmetric cone) is self-dual.
\end{proof}

Since a Euclidean Jordan algebra is complete and its hyperbolicity cone is self-dual, we obtain the following corollary.

\begin{corollary}
    Suppose $\V$ is a Euclidean Jordan algebra with unit $e$ and rank $n$.
    Then the following are equivalent for a $\V$-matrix $\bfA = \big[ a_1, a_2, \dots, a_n \big]$:
    \begin{itemize}
        \item[$(i)$] $\bfA$ is $e$-doubly stochastic.
        \item[$(ii)$] $\lambda(\bfA \bfr) \prec \bfr$ for all $\bfr \in \Rn$.
        \item[$(iii)$] $\bfA^{T} x \prec \lambda(x)$ for all $x \in \V$.
    \end{itemize}
\end{corollary}

\begin{remark}
    Suppose $\V$ is a Euclidean Jordan algebra with unit $e$ and rank $n$. Let $\{e_1, e_2, \dots, e_n\}$ be a Jordan frame in $\V$. Relative to this Jordan frame, any $x \in \V$ has a Peirce decomposition \cite{faraut-koranyi}:
    \[ x = \sum_{i=1}^{n} x_i e_i + \sum_{i < j} x_{ij}, \]
    where $x_i \in \R$ and $x_{ij} \in \V(e_i, \frac{1}{2}) \cap \V(e_j, \frac{1}{2})$; see \cite[Theorem IV.2.1]{faraut-koranyi}. From the orthogonality of elements in a Jordan frame, 
    $\ip{x}{e_i} = x_i$ for all $1 \leq i \leq n$. As $\bfA= \big[ e_1, e_2, \dots, e_n \big]$ is $e$-doubly stochastic, the statement $(iii)$ in the above corollary, $\bfA^{T} x \prec \lambda(x)$, can be interpreted as ``The diagonal of $x$ (in the Peirce decomposition relative to a Jordan frame) is majorized by the eigenvalue vector of $x$.'' This statement can be regarded as a generalization of Schur's theorem that the diagonal of a Hermitian matrix is majorized by its eigenvalue vector; see \cite[Corollary 4.6]{gowda-tao-cauchy}.
\end{remark}

Our next example shows that the converse of Theorem \ref{majorization for v-matrices2} may fail without the assumption $\lplus = \lplus^{\ast}$.

\begin{example}
    Let $\V$ be a $5$-dimensional subspace of $\mathcal{S}^3$ given by
    \[ \V = \set{\begin{bmatrix} x_1 & x_2 & x_3 \\ x_2 & x_4 & 0 \\ x_3 & 0 & x_5
    \end{bmatrix}}{x_1, \dots, x_5 \in \R}. \]
    Then $(\V, \det, I)$ is a complete hyperbolic system with the corresponding hyperbolicity cone $\lplus = \V \cap \mathcal{S}_+^3$. Its dual cone  (with respect to the trace inner product)
    \[ \lplus^{\ast} = \set{\begin{bmatrix} x_1 & x_2 & x_3 \\ x_2 & x_4 & 0 \\ x_3 & 0 & x_5
    \end{bmatrix}}{\begin{bmatrix} x_1 & x_2 \\ x_2 & x_4
        \end{bmatrix} \in \mathcal{S}^2_+, \begin{bmatrix} x_1 & x_3 \\ x_3 & x_5
        \end{bmatrix} \in \mathcal{S}^2_+} \]
    is known in the literature as the \emph{Vinberg cone}. As $\lplus$ is a proper subset of $\lplus^{\ast}$, we see that $\lplus$ is not self-dual; see \cite{gouveia-ito-lourenco}. We now consider an orthogonal matrix $U$ whose columns are 
    \[ r_1 = \frac{1}{3} (1, -2, 2)^{T}, \;\; r_2 = \frac{1}{3} (2, 2, 1)^{T}, \;\; r_3 = \frac{1}{3} (-2, 1, 2)^{T}. \]
    Let $\bfA = \big[ A_1, A_2, A_3 \big]$, where $A_i = \proj_{\V}(r_i r_i^{T})$ (the projection of $r_i r_i^{T}$ onto $\V$) for $i = 1, 2, 3$. Explicitly,
    \[ A_1 = \frac{1}{9} \begin{bmatrix*}[r] 1 & -2 & 2 \\ -2 & 4 & 0 \\ 2 & 0 & 4 \end{bmatrix*}\!, \;\; 
    A_2 = \frac{1}{9} \begin{bmatrix*}[r] 4 & 4 & 2 \\ 4 & 4 & 0 \\ 2 & 0 & 1 \end{bmatrix*}\!, \;\;
    A_3 = \frac{1}{9} \begin{bmatrix*}[r] 4 & -2 & -4 \\ -2 & 1 & 0 \\ -4 & 0 & 4 \end{bmatrix*}\!. \]
    Now let $X\in \V$. Since the $(2, 3)$ and $(3, 2)$-entries of $X$ are zero, we have
    \[ \ip{A_i}{X} = \ip{\proj_{\V}(r_i r_i^{T})}{X} = \ip{r_i r_i^{T}}{X} = \tr(r_i r_i^{T} X) = r_i^{T} X r_i. \]
    Hence, by the Schur majorization theorem (that the diagonal of a Hermitian matrix is majorized by its eigenvalue vector), see \cite{marshal-olkin},
    \[ \bfA^{T} X = \operatorname{diag}(U^{T} X U) \prec \lambda(U^{T} X U) = \lambda(X). \]
  However, as $\det(A_1)<0$, $A_1\not \in \lplus$; thus, $\bfA$ is not $e$-doubly stochastic.
\end{example}

We recall (from Section 2.3) that a linear transformation $T : \V \to \V$ is \emph{$\lambda$-doubly stochastic} if 
\[ \lambda(T(x)) \prec \lambda(x) \text{ for all } x \in \V. \]

\begin{theorem} \label{cor:operator_DS}
    \textup{(A generalization of Schur's theorem)} 
    Suppose $\bfA = \big[ a_1, a_2, \dots, a_n \big]$ and $\bfB = \big[ b_1, b_2, \dots, b_n \big]$ are $e$-doubly stochastic. Then $\bfA \odot \bfB$ is a $\lambda$-doubly stochastic transformation. Consequently, when $p$ is complete, $\bfA \odot \bfB$ is an operator doubly stochastic transformation.
\end{theorem}

\begin{proof}
    Let $x \in \V$. Then by Definition \ref{v-matrix defn}, $(\bfA \odot \bfB)(x) = \sum_{i=1}^{n} \ip{b_i}{x} a_i$. Now, letting $\bfr \coloneq \bfB^{T} x =  (\ip{b_1}{x}, \ip{b_2}{x}, \dots, \ip{b_n}{x})$ we have $(\bfA \odot \bfB)(x) = \bfA \bfr$. This shows that
    \[ \lambda((\bfA \odot \bfB)(x)) = \lambda(\bfA \bfr) \prec \bfr = \bfB^{T} x \prec \lambda(x), \]
    where the majorization inequalities come from Theorems \ref{majorization for v-matrices1} and \ref{majorization for v-matrices2}, respectively. Hence, $\bfA \odot \bfB$ (as a transformation) is $\lambda$-doubly stochastic. Consequently, when $p$ is complete, it is operator doubly stochastic; see (\ref{convex to lambda to operator ds}).
\end{proof}

\begin{remark}
    With $(\V, p, e) = (\Hn, \det, I_{n})$, we can let $\bfA = \bfB = \big[ E_1, E_2, \dots, E_n \big]$, where $E_i$ denotes the $n \times n$ matrix with $1$ in the $(i,i)$ slot and zeros elsewhere. Then Theorem \ref{cor:operator_DS} reduces to Schur's well-known result that the diagonal of a Hermitian matrix is majorized by its eigenvalue vector. Additionally, the above result solves a problem posed in \cite{gowda-jeong-shukla1} that asks whether $\bfA\odot \bfJ$ is $\lambda$-doubly stochastic, where $\bfJ$ is a $\JF$-matrix.
\end{remark}

In the following results, we specify some conditions on $\bfA$ (and/or $\bfB$) for which the converse in the (first part of the) above theorem holds.

\begin{theorem}
    Consider a complete hyperbolic system $(\V, p, e)$ of degree $n > 1$. For a $\V$-matrix $\bfA = \big[a_1, a_2, \dots, a_n \big]$, suppose $\bfA \odot \bfB$ is a $\lambda$-doubly stochastic transformation for every $e$-doubly stochastic $\V$-matrix $\bfB$. Then $\sum_{i=1}^{n}a_i = e$. Moreover, $\bfA$ is $e$-doubly stochastic when $\dim(\V)>1$ and $a_k \in \lplus$ for all $k$. In particular, $\bfA$ is $e$-doubly stochastic when $\V$ has a Jordan frame.
\end{theorem}

\begin{proof}
    Since $p$ is assumed to be complete, $\bfA \odot \bfB$ is operator doubly stochastic (thus, positive, unital, and trace-preserving) for all $e$-doubly stochastic $\bfB$. In particular, as $ \bfE = \big[ \frac{1}{n} e, \frac{1}{n} e, \dots, \frac{1}{n} e \big]$ is $e$-doubly stochastic, $\bfA \odot \bfE$ is unital. Hence, 
    \[ e = (\bfA \odot \bfE)(e) = \frac{1}{n} \sum_{i=1}^{n} \ip{e}{e} a_i =  \sum_{i=1}^{n} a_i. \]

    Next, suppose $\bfB = \big[ b_1, b_2, \ldots, b_n \big]$ is $e$-doubly stochastic. Then, $\bfA \odot \bfB$ is trace-preserving and so
    \[ \tr(x) = \tr \big( (\bfA \odot \bfB) x \big) = \sum_{i=1}^{n} \ip{b_i}{x} \tr(a_i) = \ip{\sum_{i=1}^{n} \tr(a_i) b_i}{x}, \]
    for any $x \in \V$; consequently,  by the completeness of $p$,
    \begin{equation} \label{eq:odot_converse}
        \sum_{i=1}^{n} \tr(a_i) b_i = e.
    \end{equation}

    Now, assume that $\dim(\V)>1$ and $a_k \in \lplus$ for all $k$. As $\dim(\V)>1$, from the Remark stated after Proposition \ref{delta is compact convex}, we have $\Delta \neq \{ \frac{1}{n}e \}$. Fix $b \in \Delta$ with $b \neq \frac{1}{n}e$ and consider $\bfB_1 = \big[ b, \frac{1}{n-1}(e - b), \dots, \frac{1}{n-1}(e - b) \big]$, which is $e$-doubly stochastic as noted in Example \ref{example-SG}. Thus, \eqref{eq:odot_converse} holds for this $\bfB_1$. Using the fact that $\sum_{i=1}^{n} \tr(a_i) = \tr(e) = n$, we get
    \[ e = \tr(a_1) b + \sum_{i=2}^{n}  \frac{\tr(a_i)}{n - 1}(e - b) = \tr(a_1) b + \frac{n - \tr(a_1)}{n - 1}(e - b). \]
    This simplifies to $(\tr(a_1) - 1)(nb - e) = 0$. Since $b \neq \frac{1}{n}e$, we get $\tr(a_1) = 1$. By repeating this argument with $\bfB \coloneq \bfB_k$, whose $k$th `column' is $b$ and $\frac{1}{n-1}(e - b)$ elsewhere, we get $\tr(a_k) = 1$ for all $k$. By our assumption, $a_k\in \lplus$ for all $k$. Thus, $\bfA$ is $e$-doubly stochastic.

    Now suppose $\V$ has a Jordan frame, say $\{c_1, c_2, \dots, c_n\}$, so that $\bfJ = \big[ c_1, c_2, \dots, c_n \big]$ is a $\JF$-matrix. Again, as $n > 1$, it follows that $\dim(\V) > 1$ by Proposition \ref{orthonormality of JF}. Moreover, for each $k$, the orthonormality of a Jordan frame gives
    \[ (\bfA \odot \bfJ)(c_k) = \sum_{i=1}^{n} \ip{c_i}{c_k} a_i = a_k \]
    Then, as $\bfA \odot \bfJ$ is $\lambda$-doubly stochastic, $\lambda(a_k) = \lambda \big( (\bfA \odot \bfJ)(c_k) \big) \prec \lambda(c_k) = (1, 0, \dots, 0)^{T}$. This implies $\lambda(a_k) \geq 0$; hence $a_k \in \lplus$ for all $k$.
\end{proof}

\begin{remark}
    The above result may fail without the assumption that $\dim(\V) > 1$. For example, consider $\V = \R$, $p(x) = x^2$, and $e = 1$ so that $\Delta = \{ \frac{1}{2} \}$ and $\tr(x) = 2x$. In this system, the only $e$-doubly stochastic matrix is $\bfB = \big[ \frac{1}{2}, \frac{1}{2} \big]$. The $\V$-matrix $\bfA = [1, 0]$ is clearly not $e$-doubly stochastic, yet $\bfA \odot \bfB$ is $\lambda$-doubly stochastic.
\end{remark}

\begin{theorem}
    Consider a hyperbolic system $(\V, p, e)$ having a Jordan frame with $p$ complete and $\lplus = \lplus^{\ast}$. Suppose  $\bfB = \big[ b_1, b_2, \dots, b_n \big]$ is a $\V$-matrix such that $\bfA \odot \bfB$ is a $\lambda$-doubly stochastic transformation for every $e$-doubly stochastic $\V$-matrix $\bfA$. Then $\bfB$ is $e$-doubly stochastic.
\end{theorem}

\begin{proof}Assume that the conditions are in place. Let $\bfJ = \big[ c_1, c_2, \dots, c_n \big]$ be a $\JF$-matrix (coming from a Jordan frame in $\V$). Since $\bfJ$ is $e$-doubly stochastic, $\bfJ \odot \bfB$ is $\lambda$-doubly stochastic. Hence, for every $x \in \V$, by \cite[Theorem 4.16]{gowda-jeong-shukla1}, 
    \[ \big[ \bfB^{T}x \big]^{\downarrow}=\lambda \big( \bfJ (\bfB^{T}x) \big)=\lambda \big( (\bfJ \odot \bfB) x \big) \prec \lambda(x). \]
    It follows that $\bfB^{T} x \prec \lambda(x)$ for all $x \in \V$. Since $p$ is assumed to be complete and $\lplus = \lplus^{\ast}$, by Theorem \ref{majorization for v-matrices2}, $\bfB$ is $e$-doubly stochastic.
\end{proof}

\begin{remark}
    In the above theorem, the existence of a Jordan frame cannot be dropped. To see this, consider $\V = \R$, $p(x) = x^2$, and $e = 1$. Then $(\V, p, e)$ is a complete hyperbolic system with $\lambda(x) = (x, x)^{T}$ and $\lplus = \lplus^{\ast} = \R_+$. In this system, the only $e$-doubly stochastic $\V$-matrix is $\bfA = \big[ \frac{1}{2} , \frac{1}{2} \big]$. Now, consider the $\V$-matrix $\bfB = \big[ 1, 0 \big]$ which is not $e$-doubly stochastic. However, $\bfA \odot \bfB$ is the identity transformation, hence $\lambda$-doubly stochastic.
\end{remark}

\section{The set of \texorpdfstring{$e$}{e}-doubly stochastic \texorpdfstring{$\V$}{V}-matrices}

Given a hyperbolic system $(\V, p, e)$ of degree $n$, let $\V^n$ denote the Cartesian product of $n$-copies of $\V$. When $p$ is complete, $\V$ carries the inner product induced by $\lambda$; correspondingly, $\V^n$ becomes a Hilbert space. In this setting, we let 
\begin{align*}
    \DS_n(\V) &\coloneq \set{\bfA = \big[ a_1, a_2, \dots, a_n \big]}{\text{$\bfA$ is $e$-doubly stochastic}}, \\
    \JF_n(\V) &\coloneq \set{\bfC = \big[ c_1, c_2, \dots, c_n \big]}{\text{$\bfC$ is a $\JF$-matrix}}, \\
    \GJF_n(\V) &\coloneq \set{\bfC = \big[ c_1, c_2, \dots, c_n \big]}{\text{$\bfC$ is a $\GJF$-matrix}}.
\end{align*}
Let $\Ext(\mathrm{DS}_n(\V))$ denote the set of all extreme points of the convex set $\mathrm{DS}_n(\V)$.

\begin{theorem} \label{thm:doubly stochastic transformations}
    In a complete hyperbolic system $(\V, p, e)$, the following statements hold:
    \begin{itemize}
        \item[$(a)$] $\DS_n(\V)$ is nonempty, compact and convex in $\V^n$.
        \item[$(b)$] $\JF_n(\V) = \set{\bfA \in \DS_n(\V)}{\exists\, \bfB \in \DS_n(\V), \, \bfA \ast \bfB = I_{n}}$. 
        \item[$(c)$] $\JF_n(\V)\subseteq \GJF_n(\V) \subseteq \Ext( \DS_n(\V))$.
    \end{itemize}
\end{theorem}

\textit{Note}: While $\JF_n(\V)$ may be empty, the set $ \GJF_n(\V)$ is always nonempty; the second inclusion in $(c)$ may be proper; see examples below. As $\V$ is finite dimensional, by the well-known theorem of Minkowski, see \cite[Theorem 17.2]{rockafellar}, 
\[ \conv \! \big(\Ext(\mathrm{DS}_n(\V)) \big)=\mathrm{DS}_n(\V). \]

\begin{proof}
    $(a)$ The nonemptiness comes from Example \ref{n=1 case}. The convexity of $\mathrm{DS}_n(\V)$ is easily verified. Since $p$ is complete, $\V$ is a (finite-dimensional) Hilbert space under the inner product in \eqref{ip}; consequently, $\V^n$ is a Hilbert space. By the continuity of the trace functional, it is easy to see that $\mathrm{DS}_n(\V)$ is closed in $\V^n$. If $a\in \lplus$ with trace one, then
    \[ \norm{a}^2 = \norm{\lambda(a)}^2 \leq \Big( \sum_{i=1}^{n}\lambda_i(a) \Big)^2=1. \]
    It follows that $\mathrm{DS}_n(\V)$ is a bounded set in $\V^n$; consequently, $\mathrm{DS}_n(\V)$ is compact.
    
    $(b)$ Clearly, $\JF_n(\V) \subseteq \DS_n(\V)$. If $\bfA \in \JF_n(\V)$, then $\bfA \ast \bfA = I_{n}$ as every Jordan frame is orthogonal; see Proposition \ref{orthonormality of JF}. So we have one inclusion in $(b)$. To see the reverse inclusion, suppose $\bfA \in \DS_n(\V)$ such that for some $\bfB = \big[ b_1, b_2, \dots, b_n \big] \in \DS_n(\V)$, $\bfA \ast \bfB = I_{n}$. 
    From Theorem \ref{majorization for v-matrices2}, $\bfA^{T} x \prec \lambda(x)$ for all $x \in \V$. As $\ip{a_1}{b_1} = 1$ and $\ip{a_i}{b_1} = 0$ for all $i \neq 1$, we have 
    \[ (1, 0, \dots, 0)^{T} = \bfA^{T} b_1 \prec \lambda(b_1). \]
    Since $b_1 \in \lplus$ with $\tr(b_1) = 1$, this yields $\lambda(b_1) = (1, 0, \dots, 0)^{T}$. A similar argument shows that $\lambda(b_k) = (1, 0, \dots, 0)^{T}$ for all $k$. As $\bfB$ is $e$-doubly stochastic, $b_1 + b_2 + \dots + b_n = e$. Thus, $\{b_1, b_2, \dots, b_n\}$ forms a Jordan frame. Reversing the roles of $\bfA$ and $\bfB$, we see that $\bfA$ also comes from a Jordan frame. Thus, we have the equality in $(b)$.

    $(c)$ The inclusion $\JF_n(\V)\subseteq \GJF_n(\V)$ has been observed in an earlier remark. Now let $\bfC = \big[ c_1, c_2, \dots, c_n \big] \in \GJF_n(\V)$; without loss of generality, we assume that $\bfC$ comes from the construction in Example \ref{example via extreme points}. We show that $\bfC$ is an extreme point of $\DS_n(\V)$. Suppose
    \[ \bfC = \frac{1}{2}(\bfA + \bfB), \]
    where $\bfA = \big[ a_1, a_2, \dots, a_n \big] \in \mathrm{DS}_n(\V)$ and $\bfB =\big[ b_1, b_2, \dots, b_n \big] \in \mathrm{DS}_n(\V)$. Then $c_i = \frac{1}{2}(a_i + b_i)$ for all $1 \leq i \leq n$. Now, as $\bfC \in \GJF_n(\V)$, $c_1\in \Ext(\Delta)$. Since $a_1,b_1\in \Delta$, we must have $c_1=a_1=b_1$. Then, from \eqref{useful inequality},
    \[ 0 \leq a_2 \leq e - a_1 = e - c_1 \;\; \text{and} \;\; 0 \leq b_2 \leq e - b_1 = e - c_1 \] 
    and so, $a_2, b_2 \in \Delta(c_1)$. Since $c_2$ is an extreme point of $\Delta(c_1)$, we must have $c_2 = a_2 = b_2$. By continuing this process (or by induction), $c_i = a_i = b_i$ for every $1 \leq i \leq n$. This proves that $\bfC = \bfA = \bfB$. Thus, $\bfC$ is an extreme point of $\DS_n(\V)$. 
\end{proof}

\begin{example} (The classical case) \\
    In the classical case of $\V = \Rn$ with $p(x) = x_1 x_2 \cdots x_n$, and $e = \one$, $\mathrm{DS}_n(\Rn)$ is the set of all doubly stochastic matrices. Since there is only one Jordan frame (up to permutation), $\Ext(\mathrm{DS}_n(\Rn))$ is the set of all permutation matrices. In this case, thanks to the result of Birkhoff, see the Introduction, we have equalities in Item $(c)$ of the above theorem.
\end{example}

\begin{example} (The Jordan spin algebra) \\
    On $\Rn$ with $n > 1$, consider the Euclidean Jordan algebra $\Ln$ (familiarly known as the Jordan spin algebra); see \cite{faraut-koranyi, gowda-sznajder-tao}. Writing any element $x\in \Rn$ in the form $x=(x_0, \overline{x})^{T}$, where $\overline{x} = (x_1, x_2, \dots, x_{n-1})$, we define the polynomial 
    \[ p(x) \coloneq x_0^2 - \norm{\overline{x}}_2^2, \]
    which is hyperbolic in the direction $e \coloneq (1, 0, \dots, 0)^{T}$. In $\Ln$, now identified as the hyperbolic system $ (\Rn, p, e)$, the eigenvalue map $\lambda(x) = \big( x_0 + \norm{\overline{x}}_2, x_0 - \norm{\overline{x}}_2 \big)^{T}$ induces the hyperbolicity cone $\lplus = \set{x \in \Rn}{x_0 \geq \norm{\overline{x}}_2}$. Moreover, from \eqref{ip}, we have $\ip{x}{y} = 2\, x^{T} y$. Consequently, $\tr(x) = 2x_0$ and an element in $\lplus$ with trace one is of the form 
    \[ a = \frac{1}{2}(1, \overline{a})^{T}, \text{ where }  \norm{\overline{a}}_2 \leq 1. \] 
    Such an $a$ becomes a primitive idempotent if $\norm{\overline{a}}_2 = 1$.
    
    Now, let $\bfA = \big[ a_1, a_2 \big]$ be $e$-doubly stochastic. Then, $a_1 + a_2 = e$ implies that $a_1 = \frac{1}{2}(1, \overline{a})^{T}$ and $a_2 = \frac{1}{2}(1, -\overline{a})^{T}$ for some $\overline{a} \in \R^{n-1}$ with $\norm{\overline{a}}_2 \leq 1$. In $\R^{n-1}$, we can write $\overline{a} = \alpha u + \beta v$, where $\alpha, \beta \geq 0$, $\alpha + \beta = 1$, and $\norm{u}_2 = \norm{v}_2 = 1$.
    Letting 
    \[ e_1 \coloneq \frac{1}{2}(1, u)^{T}, \; e_2 \coloneq \frac{1}{2}(1,-u)^{T}, \;\; \text{and} \;\;
    f_1 \coloneq \frac{1}{2}(1, v)^{T}, \; f_2 \coloneq \frac{1}{2}(1,-v)^{T}, \] 
    we easily verify that $\{e_1, e_2\}$ and $\{f_1, f_2\}$ are Jordan frames in this Euclidean Jordan algebra; moreover, $\bfA = \big[ a_1, a_2 \big] = \alpha \big[ e_1, e_2 \big] + \beta \big[ f_1, f_2 \big]$, which is a convex combination of extreme points of $\DS_2(\Ln)$. Thus, in this setting of $\Ln \equiv (\Rn, p, e)$, 
    \[ \Ext(\DS_2(\Ln)) = \set{\bfA = \big[ e_1, e_2 \big]}{\text{$\{e_1, e_2\}$ is a Jordan frame in $\Ln$}} = \JF_2(\Ln). \]
\end{example}

\begin{example} (The algebras $\mathcal{S}^2$ and $\mathcal{H}^2$) \\
    The Euclidean Jordan algebra $\mathcal{S}^2$ is isomorphic to $\mathcal{L}^3$. Based on the result in the previous example, we can describe the extreme points of $\DS_2(\mathcal{S}^2)$. A different proof methodology yields the following result for $\DS_2(\mathcal{H}^2)$:
    \[ \Ext(\DS_2(\mathcal{H}^2)) = \set{\bfA = \big[ uu^{\ast}, vv^{\ast} \big]}{ u, v \in \C^2, \norm{u} = \norm{v} = 1, u \perp v} = \JF_2(\mathcal{H}^2). \]
    This is seen as follows. Clearly, $\JF_2(\mathcal{H}^2) \subseteq \Ext(\DS_2(\mathcal{H}^2))$. To see the other inclusion, let $\bfA = \big[ A_1, A_2 \big]$ be $e$-doubly stochastic in $\mathcal{H}^2$. As $A_1$ and $A_2$ are positive semidefinite with trace one and $A_1 + A_2=I$, there is a unitary matrix $U$ with columns $u$ and $v$ which simultaneously diagonalizes $A_1$ and $A_2$. Easy calculations and simplifications show that 
    \[ U^{\ast} A_1 U = \begin{bmatrix} \alpha & 0 \\ 0 & \beta \end{bmatrix} \;\; \text{and} \;\; U^{\ast} A_2 U = \begin{bmatrix} \beta & 0 \\ 0 & \alpha \end{bmatrix}\!, \]
    where $\alpha, \beta \geq 0$ and $\alpha + \beta = 1$. Then, $A_1 = \alpha u u^{\ast} + \beta v v^{\ast}$ and $A_2 = \beta uu^{\ast} + \alpha vv^{\ast}$. We see that $\bfA = \alpha \big[ uu^{\ast}, vv^{\ast} \big] + \beta \big[vv^{\ast}, uu^{\ast} \big]$. Since $\big[ uu^{\ast}, vv^{\ast} \big], \big[ vv^{\ast}, uu^{\ast} \big] \in \JF_2(\mathcal{H}^2)$, $\bfA$ is a convex combination of two matrices in $\JF_2(\mathcal{H}^2)$. Hence every extreme point of $\DS_2(\mathcal{H}^2)$ is in $\JF_2(\mathcal{H}^2)$.
\end{example}

As observed in Proposition \ref{In EJA, JF=GJF}, in the setting of a Euclidean Jordan algebra, $\JF_n(\V) = \GJF_n(\V)$. We now provide examples to show that $\DS_n(\V)$ may have extreme points other than those coming from $\GJF_n(\V)$.

\begin{example} \label{product EJA example}
    Consider a Euclidean Jordan algebra $\V$ of rank $n$ and unit $e$, and carrying the trace inner product. We assume that there is a set $\{p_1, p_2, \dots, p_{2n}\}$
    of primitive idempotents that is linearly independent with 
    \[ p_1 + p_2 + \dots + p_{2n} = 2e. \]
    (For such a set in $\Ln$ and $\mathcal{H}^2$; see the remark below and the next example.)
    
    Let $\W \coloneq \V \times \V$ be the product Euclidean algebra (which has rank $2n$ with $(e, e)$ as the unit element). In $\W$, let 
    \[ a_i \coloneq \frac{1}{2}(p_i, p_i) \text{ for all } 1\leq i \leq 2n \]
    and $\bfA \coloneq \big[a_1, a_2, \dots, a_{2n} \big]$. We now claim that $\bfA \in \Ext(\DS_{2n}(\W))$, but $\bfA \notin \GJF_{2n}(\W)$.  We verify these in the following.
    
    $(i)$ In $\W$, $a_i \geq 0$ and $\tr(a_i) = 1$ for all $i$; moreover, $a_1 + a_2 + \dots + a_{2n} = (e, e)$. Thus, $\bfA$ is $e$-doubly stochastic. \\
    $(ii)$ We note that for each $i$, $(p_i, 0)$ and $(0, p_i)$ belong to $\Delta(\W)$ and $a_i$ is the average of these two elements. Hence, $a_i \not \in \Ext(\Delta(\W))$ for every $i$. It follows that $\bfA \notin \GJF_{2n}(\W)$. (Note that when $\bfA \in \GJF_{2n}(\W)$, we must have  $a_i\in \Ext(\Delta(\W))$ for some $i$.)\\
    $(iii)$ Suppose $\bfA = \frac{1}{2}(\bfB + \bfC)$, where $\bfB, \bfC \in \DS_{2n}(\W)$. Write $\bfB = \big[ b_1, b_2, \dots, b_{2n} \big]$ and $\bfC = \big[ c_1, c_2, \dots, c_{2n} \big]$, where 
    $b_i = (b_{i1}, b_{i2}), c_i = (c_{i1}, c_{i2})$ for each $i = 1, 2, \dots, 2n$. We show that $b_{i1} = c_{i1} = \frac{1}{2}p_i$ and $b_{i2} = c_{i2} = \frac{1}{2}p_i$ for all $i$. \\
    First, fix an index $i$, $1 \leq i \leq 2n$, and consider $a_i = \frac{1}{2}(b_i + c_i)$. Then (by considering the first components) $p_i = b_{i1} + c_{i1}$. As $b_{i1}, c_{i1} \geq 0$ in $\V$ and $p_i$ is a primitive idempotent (so is an extreme direction of the symmetric cone in $\V$), $b_{i1} = t_i p_i$ and $c_{i1} = (1-t_i) p_i$ for some scalar $0 \leq t_i \leq 1$. Now,
    \[ \frac{1}{2} \sum_{i=1}^{2n} p_i = e = \sum_{i=1}^{2n} b_{i1} = \sum_{i=1}^{2n} t_i p_i. \]
    By the imposed linear independence of the $p_i$s, we see that $t_i = \frac{1}{2}$ for all $i$. This implies that $b_{i1} = \frac{1}{2}p_i$. A similar statement holds for $b_{i2}$. We conclude that $b_i = (b_{i1}, b_{i2}) = \frac{1}{2}(p_i, p_i) = a_i$. As this holds for all $i$, we have $\bfB = \bfA$. Similarly, $\bfC = \bfA$. Thus, $\bfA$ is an extreme point.
\end{example}

In the above example, we constructed an extreme point of $\DS_{2n}(\W)$ that is not in $\GJF_{2n}(\W)$. Here, $\W$ is a product (so non-simple) Euclidean Jordan algebra. In what follows, we construct an explicit example in the simple Euclidean Jordan algebra $\mathcal{H}^4$. (While we skip the details, we can carry out a similar construction in ${\cal S}^6$. The proof technique used below is reminiscent of a similar construction in \cite{bapat}.)
In preparation for the example, first we state the following lemmas, where for a matrix $A$ in $\Hn$, we write $A\geq 0$ when $A$ is positive semidefinite.

\begin{lemma} \label{lemma1}
    Let $A, B \in \Hn$. If $A \geq 0$ and there exists $\varepsilon > 0$ such that $A \pm \varepsilon B \geq 0$, then $\ker(A) \subseteq \ker(B)$.
\end{lemma}

\begin{proof}
    Suppose $x \in \ker(A)$, i.e., $Ax = 0$. Then
    \[ 0 \leq x^{\ast} (A \pm \varepsilon B) x = x^{\ast}Ax \pm \varepsilon x^{\ast} Bx = \pm \varepsilon x^{\ast} Bx. \]
    Thus, we get $x^{\ast}Bx = 0$ and $x^{\ast}(A \pm \varepsilon B)x = 0$. Since $A \pm \varepsilon B \geq 0$, we must have $(A \pm \varepsilon B)x = 0$. As $Ax=0$, we must have $Bx = 0$, that is, $x \in \ker(B)$. 
\end{proof}

\begin{lemma}\label{lemma2}
    In $\C^n$ with the usual inner product, let $u \in \C^n$ be a unit vector and $D \in \C^{n \times n}$. If $Dw = D^{\ast}w = 0$ for all $w \perp u$, then $D$ is a (complex) multiple of $uu^{\ast}$.
\end{lemma}

\begin{proof} 
    Suppose $Dw = D^{\ast}w = 0$ for all $w \perp u$. We first note that any element $x \in \C^n$ can be written as $x = \alpha u + v$ for some $v \perp u$ and $\alpha \in \C$. As $\ip{Du}{w} = \ip{u}{D^{\ast} w} = 0$ for all $w \perp u$, $Du$ is a multiple of $u$; so $Du = \beta u$ for some $\beta \in \C$. Since $u^{\ast}x = \ip{x}{u} = \alpha \ip{u}{u} = \alpha$ and $Dv = 0$,
    \[ Dx = \alpha Du + Dv = \alpha \beta u = \beta u (u^{\ast}x) = \beta (uu^{\ast})x. \]
    As $x \in \C^n$ is arbitrary, it follows that $D = \beta uu^{\ast}$.
\end{proof}

\begin{example} \label{eq: counterexample in H^4}
    Consider the Euclidean Jordan algebra $\V = \mathcal{H}^4$. We show that
    \[ \JF_4(\mathcal{H}^4)= \GJF_4(\mathcal{H}^4)\neq \Ext (\DS_4(\mathcal{H}^4)). \]
    
    The equality $\JF_4(\mathcal{H}^4)= \GJF_4(\mathcal{H}^4)$
    comes from Proposition \ref{In EJA, JF=GJF}. To show $\GJF_4(\mathcal{H}^4)\neq \Ext (\DS_4(\mathcal{H}^4))$, we proceed as follows. Consider the following matrices in $\mathcal{H}^2$ (each of the form $uu^{\ast}$, where $u$ is a unit vector):
    \begin{align*}
        P_1 &= \frac{1}{4} \begin{bmatrix} 3 & \sqrt{3} \\ \sqrt{3} & 1 \end{bmatrix}\!, &
        P_2 &= \frac{1}{4} \begin{bmatrix} 3 & -\sqrt{3} \\ -\sqrt{3} & 1 \end{bmatrix}\!, \\ 
        P_3 &= \frac{1}{4} \begin{bmatrix} 1 & -i\sqrt{3} \\ i\sqrt{3} & 3 \end{bmatrix}\!, & 
        P_4 &= \frac{1}{4} \begin{bmatrix} 1 & i\sqrt{3} \\ -i\sqrt{3} & 3 \end{bmatrix}\!.
    \end{align*}
    
    One  can check that 
    \[ P_k^2 = P_k, \,\,\tr(P_k) = 1,\, \text{and}\,\,P_1 + P_2 + P_3 + P_4 = 2I_2. \]
    Moreover, a simple computation shows that the Gram matrix $G=[G_{ij}]$, given by $G_{ij} = \tr(P_i P_j)$, has nonzero determinant. It follows that the set $\{P_1, P_2, P_3, P_4\}$ in $\mathcal{H}^2$ is linearly independent  (over both $\R$ and $\C$).
    
    Now let $\bfA = \big[ A_1, A_2, A_3, A_4 \big]$, where $A_k = \dfrac{1}{2} (P_k \oplus P_k)$. Explicitly,
    \begin{align*}
        A_1 &= \frac{1}{8} \begin{bmatrix} 3 & \sqrt{3} & 0 & 0 \\ \sqrt{3} & 1 & 0 & 0 \\ 0 & 0 & 3 & \sqrt{3} \\ 0 & 0 & \sqrt{3} & 1 \end{bmatrix}\!, &
        A_2 &= \frac{1}{8} \begin{bmatrix} 3 & -\sqrt{3} & 0 & 0 \\ -\sqrt{3} & 1 & 0 & 0 \\ 0 & 0 & 3 & -\sqrt{3} \\ 0 & 0 & -\sqrt{3} & 1 \end{bmatrix}\!, \\ 
        A_3 &= \frac{1}{8} \begin{bmatrix} 1 & -i\sqrt{3} & 0 & 0 \\ i\sqrt{3} & 3 & 0 & 0 \\ 0 & 0 & 1 & -i\sqrt{3} \\ 0 & 0 & i\sqrt{3} & 3 \end{bmatrix}\!, & 
        A_4 &= \frac{1}{8} \begin{bmatrix} 1 & i\sqrt{3} & 0 & 0 \\ -i\sqrt{3} & 3 & 0 & 0 \\ 0 & 0 & 1 & i\sqrt{3} \\ 0 & 0 & -i\sqrt{3} & 3 \end{bmatrix}\!.
    \end{align*}
    From the properties of $P_i$'s, we see that $\bfA$ is $e$-doubly stochastic over $\mathcal{H}^4$. Moreover, as $\frac{1}{2}$ is an eigenvalue of $A_1$, $\bfA\not\in \JF_4(\mathcal{H}^4)$.
    
    We now show that $\bfA\in \Ext(\DS_4(\mathcal{H}^4))$. Suppose there exist $\bfB = \big[ B_1, B_2, B_3, B_4 \big]$ and $\varepsilon > 0$ such that $\bfA \pm \varepsilon \bfB \in \DS_4(\mathcal{H}^4)$. We claim that $\bfB=0.$ \\
    Now, the conditions $\bfA, \bfA \pm \varepsilon \bfB \in \DS_4(\mathcal{H}^4)$ imply that
    \[ A_k \pm \varepsilon B_k \geq 0, \;\; \tr(B_k) = 0 \text{ for all } k, \;\; \text{and} \;\; \sum_{k=1}^{4} B_k = 0. \]
    We fix $k$ and, to simplify the notation, write $A$, $B$, and $P$ in place of $A_k$, $B_k$, and $P_k$, respectively. 
    
    Recall that $P = uu^{\ast}$ for some unit vector $u \in \C^2$. We now write $A$ and $B$ as $2 \times 2$ block matrices as follows:
    \[ A = \frac{1}{2} \begin{bmatrix}
        P & 0 \\ 0 & P
    \end{bmatrix} = \frac{1}{2} \begin{bmatrix}
        uu^{\ast} & 0 \\ 0 & uu^{\ast}
    \end{bmatrix}\!, \;\; B = \begin{bmatrix}
        B_{11} & B_{12} \\ B_{12}^{\ast} & B_{22}
    \end{bmatrix}\!, \]
    where $B_{11}, B_{22} \in \mathcal{H}^2$ and $B_{12} \in \C^{2 \times 2}$. Let $v \in \C^2$ be orthogonal to $u$. Then
    \[ A = \frac{1}{2} \begin{bmatrix}
        uu^{\ast} & 0 \\ 0 & uu^{\ast}
    \end{bmatrix} \implies A \begin{bmatrix} v \\ 0 \end{bmatrix} = \begin{bmatrix} 0 \\ 0 \end{bmatrix} \text{ and } A \begin{bmatrix} 0 \\ v \end{bmatrix} = \begin{bmatrix} 0 \\ 0 \end{bmatrix}. \]
    Since $A \pm \varepsilon B \geq 0$, by Lemma \ref{lemma1}, $\ker(A) \subseteq \ker(B)$; thus
    \[ B_{11} v = B_{12} v = B_{12}^{\ast} v = B_{22} v = 0. \] 
    Noting that $B_{11}$ and $B_{22}$ are Hermitian, by Lemma \ref{lemma2}, every $2 \times 2$ submatrix of $B$ is a (complex) multiple of $uu^{\ast}$. Hence, we may write
    \[ B = \begin{bmatrix} \alpha uu^{\ast} & \beta uu^{\ast} \\ \overline{\beta} uu^{\ast} & \gamma uu^{\ast} \end{bmatrix}
    = \begin{bmatrix} \alpha P & \beta P \\ \overline{\beta} P & \gamma P \end{bmatrix} \]
    for suitable $\alpha, \gamma \in \R$ and $\beta \in \C$. Putting the subscripts back, we conclude that each $B_k$ is a block matrix whose blocks are scalar multiples of $P_k$. 
    
    From the condition $\displaystyle \sum_{k=1}^{4} B_k = 0$, we obtain three homogeneous equations, namely,
    \begin{equation} \label{eq:linear_combo}
        \sum_{k=1}^{4} \alpha_k P_k = 0, \;\; \sum_{k=1}^{4} \beta_k P_k = 0, \;\; \sum_{k=1}^{4} \gamma_k P_k = 0.
    \end{equation}
    Since $\{P_1, P_2, P_3, P_4\}$ is linearly independent over both $\R$ and $\C$, all the coefficients appearing in \eqref{eq:linear_combo} are zero. We conclude that $B_k=0$ for each $k$; consequently, $\bfB=0$. Hence, $\bfA$ is an extreme point of $\DS_4(\mathcal{H}^4)$. As mentioned earlier, $\frac{1}{2}$ is an eigenvalue of (each) $A_1$; thus $\bfA$ is not a $\JF$-matrix over $\mathcal{H}^4$.
\end{example}

As outlined below, the above construction can be repeated in ${\cal S}^6$.

\begin{example}
    Consider the Euclidean Jordan algebra $\V = \mathcal{S}^3$. Let
    \begin{align*}
        v_1 &= \begin{bmatrix} 1/\sqrt{2} \\ 1/\sqrt{2} \\ 0 \end{bmatrix}\!, &
        v_2 &= \begin{bmatrix} -1/\sqrt{2} \\ 1/\sqrt{2} \\ 0 \end{bmatrix}\!, &
        v_3 &= \begin{bmatrix} 1/\sqrt{2} \\ 0 \\ 1/\sqrt{2} \end{bmatrix}\!, \\
        v_4 &= \begin{bmatrix} -1/\sqrt{2} \\ 0 \\ 1/\sqrt{2} \end{bmatrix}\!, &
        v_5 &= \begin{bmatrix} 0 \\ 1/\sqrt{2} \\ 1/\sqrt{2} \end{bmatrix}\!, &
        v_6 &= \begin{bmatrix} 0 \\ -1/\sqrt{2} \\ 1/\sqrt{2} \end{bmatrix}\!.
    \end{align*}
    Let $P_i = v_i v_i^{T}$ for $i = 1, \dots, 6$. Then each $P_i$ is a primitive idempotent of $\mathcal{S}^3$, $\{P_1, P_2, \dots, P_6\}$ is linearly independent, and $\sum_{i=1}^6 P_i = 2I$. As in Example \ref{eq: counterexample in H^4}, in $\mathcal{S}^6$, consider $A_i = \dfrac{1}{2} (P_i \oplus P_i)$ for $i = 1, 2, \dots, 6$. Then $\bfA \coloneq \big[ A_1, A_2, \dots, A_6 \big]$ is an extreme point of $\DS_6(\mathcal{S}^6)$ but is not in $\GJF_6(\mathcal{S}^6)$. This can be seen, as in the previous example, by using lemmas \ref{lemma1} and \ref{lemma2}. 
\end{example}

\section{Concluding remarks} 

Motivated by the works of Bapat \cite{bapat}, Alberti and Uhlmann \cite{alberti-uhlmann-paper, alberti-uhlmann}, and Gurvits \cite{gurvits1}, in this paper we studied several properties of $e$-doubly stochastic $\V$-matrices over hyperbolic systems. Various notions of doubly stochastic transformations (introduced in Section 2.3) and their properties will be presented in the forthcoming manuscript \cite{gowda-jeong-shukla2}.

\section*{Acknowledgments}

The work of Juyoung Jeong was supported by the National Research Foundation of Korea (NRF) grant funded by the Ministry of Science and ICT (No.\,RS-2026-25475143), and by the Global-LAMP Program of the NRF grant funded by the Ministry of Education (No.\,RS-2025-25441317).


\end{document}